\documentclass[a4paper, 12pt]{article}
\usepackage{amsmath,amssymb,amsthm}
\usepackage[section]{placeins}
\usepackage{graphicx,color}
\usepackage{hyperref}
\usepackage{float}
\usepackage{subfig}
\usepackage{dsfont}
\usepackage{capt-of}
\usepackage{etex}
\usepackage{mathtools}
\usepackage{stmaryrd}
\usepackage{authblk}
\usepackage{booktabs}
\usepackage{listings}

\usepackage{epstopdf}
\usepackage{ mathrsfs }

\usepackage{lipsum} % for filler text
\usepackage{algorithm}
\usepackage{algpseudocode}
\usepackage[pagewise]{lineno}%\linenumbers

\newcommand{\J}{\mathcal{J}}

\newcommand{\ben}{\begin{equation}}
\newcommand{\een}{\end{equation}}
\newcommand{\benn}{\begin{equation*}}
\newcommand{\eenn}{\end{equation*}}

\newcommand{\R}{\mathds{R}}

\newtheorem{theorem}{Theorem}
\newtheorem{lemma}{Lemma}
\newtheorem{corollary}{Corollary}
\newtheorem{remark}{Remark}

\begin{document}

\title{A monotonicity-based globalization of the level-set method
for inclusion detection}

\author[1]{ B. Harrach  \thanks{ harrach@math.uni-frankfurt.de} }
\author[2]{H. Meftahi \thanks{houcine.meftahi@enit.utm.tn}}
\affil[1]{Department of Mathematics, Goethe University Frankfurt, Germany.}
\affil[2]{University of Jendouba/ISIKEF and  ENIT of Tunis, Tunisia.}
\providecommand{\keywords}[1]{\textbf{\textbf{Keywords:}} #1}

\date{}

\maketitle

\begin{abstract}
We focus on a geometrical inverse problem that involves recovering discontinuities in electrical conductivity based on boundary measurements. 
This problem serves as a model to introduce a shape recovery technique that merges the monotonicity method with the level-set method. 
The level-set method, commonly used in shape optimization, often relies heavily on the accuracy of the initial guess.
To overcome this challenge, we utilize the monotonicity method to generate a more precise initial guess, which is then used to initialize the level-set method. 
We provide numerical results to illustrate the effectiveness of this combined approach.\\
\end{abstract}

\section{Introduction}
The reconstruction of anomalies  in materials through non-destructive testing is becoming increasingly crucial and has generated significant interest in the field of inverse problems. This field encompasses a wide range of applications covering engineering, geosciences and medical diagnostics.

In engineering, particularly in the aerospace and automotive industries, it is essential to accurately detect and characterize defects or inclusions in materials to ensure the integrity and reliability of structures. Non-destructive testing techniques play an essential role in this respect, as they enable inspections to be carried out without compromising the integrity of the materials tested.

Numerous researchers have addressed the task of reconstructing anomalies using diverse methodologies, including iterative  algorithms based on
topological and shape gradient methods; see for instance
 \cite{belhachmi2023level,chaabane2013topological,belhachmi2013shape,meftahi2009etudes, belhachmi2018topology,laurain2016shape}
and the references therin.

The level-set method is a powerful tool for shape optimization that represents shapes implicitly as the zero level-set of a higher-dimensional function, 
typically a signed distance function. The shape is evolved over time by solving partial differential equations (PDEs) that govern the movement of the level-set.
 These methods allow for complex topological changes like merging and splitting, making them particularly useful for problems where the optimal shape may 
have unknown or complicated boundaries. The level-set method provides a flexible framework to handle large deformations and can naturally incorporate constraints 
such as volume preservation. For an overview of the level-set method the reader is referred to    \cite{belhachmi2023level,laurain2016shape} and the refernces therin.

In the context of geometrical inverse problems,  numerical methods  based on the level-set method and on the shape differentiation
are highly dependent on the initial guess. This is because these problems are often ill-posed, meaning they can have multiple solutions or no solution at all under 
certain conditions. The presence of numerous local minima, each with potentially different topologies, further complicates the problem. As a result, the choice of the initial
 guess plays a critical role in determining the success of the algorithm, as it can significantly affect whether the method converges to a desirable solution or gets trapped in 
suboptimal ones. Therefore, selecting an appropriate initial guess,  can be crucial in overcoming these challenges.

In \cite{chaabane2013topological}, the authors demonstrate that, for certain shapes, the issue of selecting an initial guess can be circumvented by employing the topological 
gradient method. In this work, we propose an alternative approach that overcomes this problem by utilizing the monotonicity property of the Neumann-to-Dirichlet operator.

Our main contribution is numerical. We develop a numerical reconstruction method that combines the monotonicity method to obtain an initial approximation of the desired solution, 
followed by refinement using the level set method. We have also shown that our numerical method can successfully recover both the shape and the parameter solutions of the inverse problem.

Let us stress that the problem of local convergence, generally applies to inverse coefficient problems in PDEs. Known reconstruction methods often rely on iterative, e.g., Newton-type approaches or on globally minimizing a non-convex regularized data-fitting functional. Such approaches only locally converge and they can be observed to heavily depend on the initial guess. In this work we benefit from having the monotonicity method as a globally convergent method for the selection of the initial guess.
More generally, the construction of globally convergent algorithms has been studied using quasi-reversibility and convexification ideas in the seminal work of
Klibanov et al., cf., e.g., \cite{beilina2008globally,beilina2012approximate,klibanov2017convexification,klibanov2019convexification}, and we also mention the recent works of one of the authors on global Newton convergence and convex semidefinite reformulations for EIT and a related Robin problem \cite{harrach2021uniqueness,harrach2021solving,harrach2023calderon}.

Let us give a brief overview about the the monotonicity method:\\
The monotonicity method was first introduced in \cite{tamburrino2002new} for addressing the electrical resistance tomography problem.
 The authors proposed a non-iterative approach based on the monotonicity properties of the resistance matrix.
 Since then, this method has been further developed for applications in static Electrical Capacitance Tomography,
 Inductance Tomography, and Electrical Resistance Tomography \cite{calvano2012fast, garde2022simplified, garde2022reconstruction}. 
It has also been extended to parabolic problems, including Magnetic Induction Tomography (MIT) \cite{Tamburrino}. 
Additionally, the monotonicity method has been applied to linear elasticity \cite{eberle2023resolution, eberle2021shape}, 
where both the standard and linearized versions were analyzed. Monotonicity properties have also been leveraged for shape reconstruction in electrical impedance tomography \cite{harrach2013monotonicity}.

The paper is organized as follows: Section 2 introduces the direct and inverse problems. Section 3 discusses the monotonicity property of the
 Neumann-to-Dirichlet operator and its application to shape reconstruction, including numerical results. Section 4 reformulates the inverse problem 
as a minimization problem using the Kohn-Vogelius functional, provides the shape derivative, and presents numerical results using the level-set method. 
Section 5 combines the monotonicity method with the level-set method for reconstruction and includes additional numerical results.

\section{A  model  problem  formulation}
We consider a bounded domain $\Omega\subset \R^2$ with smooth boundary  $\partial\Omega$ and outer unit normal vector $\nu$.  Assume that the conductivity in 
$\Omega$  is  $\sigma = \sigma_0 + (\sigma_1-\sigma_0)\chi_D$,  where  $D \in \mathcal{O}_{ad}$,  with 
\[
\mathcal{O}_{ad} :=\{D \text{  open of class } C^{1,1}, \text{ and } \overline D \subset \Omega \text{ has connected complement}\} ,
\]
and $\sigma_0, \sigma_1$ are positive constants and $\chi$ denotes the indicator function. 

Denote $L^2_{\Diamond}(\partial\Omega)$ (resp. $H^1_{\Diamond}(\Omega)$ ) the space of  $L^2$
(resp. $H^1$)-functions with  vanishing  integral  mean on $\partial\Omega$.
For a given current density $g\in L^2_{\Diamond}(\partial\Omega)$ the potential $u\in H^1_{\Diamond}(\Omega)$  satisfies the following Neumann problem
\begin{equation}
\label{EIT}
\nabla\cdot (\sigma\nabla u) = 0\text{ in } \Omega,  \quad \sigma\partial_\nu u|_{\partial\Omega} =g,\quad \int_{\partial\Omega} u\,ds=0.
\end{equation}
The weak formulation for problem \eqref{EIT} reads 
\begin{equation}
\label{EIT_f}
\int_\Omega \sigma \nabla u\cdot \nabla  w\,dx = \int_{\partial\Omega} gw\,ds, \text{ for all } w\in  H^1_{\Diamond}(\Omega).
\end{equation}
It is well known that \eqref{EIT_f} has a unique solution $u\in H^1_{\Diamond}(\Omega)$. Then the Neumann-to-Dirichlet  operator 
$\Lambda(\sigma):  L^2_{\Diamond}(\partial\Omega)\rightarrow L^2_{\Diamond}(\partial\Omega), \quad  g \mapsto u\vert_{\partial\Omega}$
is well defined. It is easily shown  that  $\Lambda(\sigma)\in \mathcal L(L^2_{\Diamond}(\partial\Omega))$ is self-adjoint  and compact.

The quadratic form associated with  $\Lambda(\sigma)$ is defined by 
\[
\langle \Lambda(\sigma)g,g \rangle= \int_\Omega \sigma \vert \nabla u\vert^2\,dx.
\]
It is also know that $\Lambda$ is Fr\' ech\' et differentiable and its    Fr\' ech\' et derivative  $\Lambda^\prime(\sigma)\in \mathcal{L}(L_\Diamond^2(\partial\Omega))$ fulfills 
\begin{equation}\label{Frechet_var_form}
\langle (\Lambda^\prime(\sigma)\kappa) g,g \rangle= -\int_\Omega \kappa \vert \nabla u\vert^2\,dx, \quad \text{ for all }\kappa \in  L^\infty(\Omega).
\end{equation}

We assume that  the values  $\sigma_0$ and $\sigma_1$  are known {\it a priori} but the
information about the geometry $D$ is missing. Then, the geometrical  inverse problem we consider here is the following:
\begin{equation}
\label{invp}
\begin{aligned}
\text{ \it  Find   the shape  $D$ }  
 \text{ \it knowing the Neuman-to-Dirichlet  map } \Lambda(\sigma). 
\end{aligned} 
 \end{equation}
 This is equivalent to  find the support  $\text{supp}(\sigma_1-\sigma_0)$ from  the Neumann-to-Dirichlet  operator $\Lambda(\sigma)$.
We will also consider the problem, where $\sigma_0$ and the sign of the contrast $\sigma_1>\sigma_0$ is known, but the value of $\sigma_1$ is also to be reconstructed
from $\Lambda(\sigma)$.

 \section{Shape recovery via the monotonicity method}
This section briefly summarizes the monotonocity method for reconstructing embedded shapes in conducting bodies. We start with the concept of Loewner monotonicity and then derive the standard (non-linearized) monotonicity method, and the faster linearized monotonicity method as in \cite{harrach2013monotonicity}. We then present the noise-robust monotonicity-based regularization method from \cite{harrach2016enhancing}.

\subsection{Loewner monotonicity}

To motivate the monotonicity-based method let us recall the following lemma.
\begin{lemma}\label{lemma:mon_inequ}
Let $\sigma, \tau\in L^\infty_{+}(\Omega)$ be two conductivities  and let 
$g\in L^2_{\diamond}(\partial\Omega)$ be an applied boundary current. 
Let $u_1= u^g_{\sigma}, u_2=u^g_{\tau}\in H^1_{\diamond}(\Omega)$. Then
\begin{equation}\label{mon_inequ_1}
\int_{\Omega}(\sigma-\tau)\vert\nabla u_\tau\vert^2\,dx\geq
\langle \left(\Lambda(\tau)-\Lambda(\sigma)\right)g,g\rangle \geq
\int_{\Omega}\frac{\tau}{\sigma}(\sigma-\tau)\vert\nabla u_\tau\vert^2\,dx,
\end{equation}
and
\begin{equation}\label{mon_inequ_2}
\int_{\Omega}(\sigma-\tau)\vert\nabla u_\sigma\vert^2\,dx\leq
\langle \left(\Lambda(\tau)-\Lambda(\sigma)\right)g,g\rangle \leq
\int_{\Omega}\frac{\sigma}{\tau}(\sigma-\tau)\vert\nabla u_\sigma\vert^2\,dx.
\end{equation}
\end{lemma}
\begin{proof}
This lemma goes back to the works \cite{kang1997inverse,ikehata1998size}. A short proof of \eqref{mon_inequ_1} using the variational formulation \eqref{EIT_f} can be found in \cite[Lemma~3.1]{harrach2013monotonicity}. The other inequalities \eqref{mon_inequ_2} then follow from interchanging $\tau$ and $\sigma$.
\end{proof}

Lemma~\ref{lemma:mon_inequ} implies monotonicity of the mapping $\sigma \mapsto \Lambda(\sigma)$ with respect to the following partial orderings. For symmetric operators $A,B\in \mathcal L(L^2_{\Diamond}(\partial\Omega))$ we introduce the semidefinite (aka Loewner) ordering:
\[
A\preceq B \quad \text{ denotes that } \quad \int_{\partial \Omega} g (B-A) g ds \geq 0 \quad \text{ for all } g\in L^2_{\Diamond}(\partial\Omega).
\]
Also, for functions $\sigma,\tau\in L^\infty(\Omega)$
\[
\sigma\preceq \tau \quad \text{ denotes that } \quad \tau(x)\geq \sigma(x) \quad \text{ for all } x\in \Omega\ \text{(a.e.)}.
\]

With this notation, we obtain the following result as an immediate consequence of Lemma~\ref{lemma:mon_inequ}.
\begin{corollary}\label{cor:monotony_simple}
Let $\sigma_1, \sigma_2\in L^\infty_{+}(\Omega)$. Then
\[
\sigma_1\leq \sigma_2\quad \text{ implies }\quad \Lambda(\sigma_1)\succeq \Lambda(\sigma_2).
\]
\end{corollary}

\subsection{Standard and linearized monotonicity method}

We return to the problem of reconstructing the shape $D\in \mathcal{O}_{ad}$ from knowing $\Lambda(\sigma)$ where $\sigma=\sigma_0+(\sigma_1-\sigma_0)\chi_D$. 
Clearly, for any open set $B\subseteq D$, and every constant $\alpha>0$ with $\alpha \leq \sigma_1-\sigma_0$, we have that 
\[
\sigma_0+\alpha \chi_B\leq \sigma \quad \text{ and thus } \quad \Lambda(\sigma_0+\alpha \chi_B)\succeq \Lambda(\sigma),
\]
by Corollary \ref{cor:monotony_simple}. Consequently, by marking all open sets $B$ fulfilling $\Lambda(\sigma_0+\alpha \chi_B)\succeq \Lambda(\sigma)$ for some $\alpha>0$, one would obtain a superset of the unknown shape $D$. The work \cite{harrach2013monotonicity} proved that $D$ can indeed be reconstructed by such monotonicity tests:

\begin{theorem}\label{thm:mono_test_standard}
Let $0<\alpha\leq \sigma_1-\sigma_0$, and $B$ be an open set. Then 
\[
B\subseteq \overline D \quad \text{ if and only if } \quad \Lambda(\sigma_0+\alpha \chi_B)\succeq \Lambda(\sigma).
\]
\end{theorem}
\begin{proof}
Since $\partial D$ is a null set, $B\subseteq \overline D$ implies that $\sigma_0+\alpha \chi_B\leq \sigma$, 
and thus $\Lambda(\sigma_0+\alpha \chi_B)\succeq \Lambda(\sigma)$ follows from Corollary \ref{cor:monotony_simple} as explained above. 
The converse implication is more involved and utilizes the idea of localized potentials from \cite{gebauer2008localized}. The details can be found in \cite[Thm~4.1]{harrach2013monotonicity}.
\end{proof}

Note that implementing the monotonicity tests in Theorem \ref{thm:mono_test_standard} would be computationally expensive since for each open set $B$
we would have to solve the EIT equation with a new inhomogeneous conductivity in order to calculate $\Lambda(\sigma_0 + \alpha\chi_B)$. 
However, using the variational formulation of the Fr\'echet derivative of $\Lambda$ in \eqref{Frechet_var_form}, one can formulate the following linearized variant of the monotonicity test, that is still capable of reconstructing the exact shape $D$.

\begin{theorem}\label{thm:mono_test_linearized}
Let $0<\alpha\leq \frac{\sigma_0}{\sigma_1} (\sigma_1-\sigma_0)=\sigma_0-\frac{\sigma_0^2}{\sigma_1}$, and $B$ be an open set. Then 
\[
B\subseteq \overline D \quad \text{ if and only if } \quad \Lambda(\sigma_0) + \alpha \Lambda'(\sigma_0) \chi_B\succeq \Lambda(\sigma).
\]
\end{theorem}
\begin{proof}
By Lemma~\ref{lemma:mon_inequ}, we obtain
\[
\langle (\Lambda(\sigma)-\Lambda(\sigma_0))g,g\rangle \leq \int_\Omega \frac{\sigma_0}{\sigma}(\sigma_0-\sigma)\vert \nabla u_{\sigma_0}\vert^2 dx
%= \int_D \frac{\sigma_0}{\sigma_1}(\sigma_0-\sigma_1)\vert \nabla u_{\sigma_0}\vert^2 dx
\leq -\alpha \int_\Omega \chi_B \vert \nabla u_{\sigma_0}\vert^2 dx.
\]
Using \eqref{Frechet_var_form}, it follows that
\[
\Lambda(\sigma)-\Lambda(\sigma_0)\preceq \alpha \Lambda'(\sigma_0) \chi_D, \quad \text{ and thus } \quad
\Lambda(\sigma_0) + \alpha \Lambda'(\sigma_0) \chi_D \succeq \Lambda(\sigma).
\]
Again, the converse implication can be proven using the idea of localized potentials, and we refer the reader for the details in \cite[Thm~4.3]{harrach2013monotonicity}.
\end{proof}

The linearized monotonicity tests can very efficiently implemented as they only require the calculation of $\Lambda'(\sigma_0) \chi_D$, which 
only requires the solution $u_{\sigma_0}$ of the conductivity equation with the background conductivity $\sigma_0$. Note that it allows for a convergent implementation for noisy data. Let $\Lambda^\delta\in \mathcal L(L^2_{\Diamond}(\partial\Omega))$ be a symmetric noisy version of $\Lambda(\sigma)-\Lambda(\sigma_0)$ with 
\[
\Vert \Lambda^\delta - \left(\Lambda(\sigma)-\Lambda(\sigma_0)\right) \Vert_{\mathcal{L}(L^2_{\diamond}(\partial\Omega))} < \delta,
\]
and $0<\alpha\leq \frac{\sigma_0}{\sigma_1} (\sigma_1-\sigma_0)$. Then, for every open set $B\subseteq \Omega$, there exists a noise level $\hat \delta>0$, such that for all $0<\delta<\hat\delta$:
\begin{equation}\label{mono_test_noise}
B\subseteq \overline D\quad \text{if and only if} \quad 
\alpha \Lambda'(\sigma_0) \chi_B\succeq \Lambda^\delta-\delta I,
\end{equation}
cf.\ \cite[Remark~3.5]{harrach2013monotonicity}, and \cite{garde2017convergence}.

\subsection{The monotonicity-based regularization method}\label{subsect:mon_regularization}

In numerical experiments (see the next subsection), the monotonicity method can be observed to be rather noise-sensitive. 
To improve its robustness, a combination with a data fitting functional has been developed in the works \cite{harrach2016enhancing,harrach2018monotonicity}.
The main idea is to minimize a discretized version of the linearized data fitting residual
\[
r:\ L^\infty(\Omega)\to \mathcal L(L^2_\Diamond(\partial \Omega)) ,\quad r(\kappa):  =  \Lambda(\sigma) -\Lambda(\sigma_0) - \Lambda^{\prime}(\sigma_0)\kappa
\]
with a constraint that is based on the monotonicity method. 

In order to explain this in detail, let $\{B_k\}_{k=1}^n$ form a partition of $\Omega$, i.e., 
\[
\overline \Omega=\bigcup_{k=1}^n \overline{B_k}, \quad \text{ with nonempty, pairwise disjoint Lipschitz domains $B_1,\ldots, B_n$.}
\]
The set of admissible parameters for minimizing the data fitting functional is defined as
\[
\mathcal P:=\left\{  \kappa\in L^\infty(\Omega):  \kappa=\sum_{k=1}^n a_k\chi_{B_k},\quad   0\leq a_k\leq  \min\{ \overline c,c_k \}\right\},
\]
where (as in Theorem \ref{thm:mono_test_linearized}) $0<\overline c\leq\sigma_0-\frac{\sigma_0^2}{\sigma_1}$, and
\[
c_k:=\max \{ \alpha>0:\ \Lambda(\sigma_0) + \alpha \Lambda'(\sigma_0) \chi_{B_k}\succeq \Lambda(\sigma)\}.
\]
It can be shown (cf.\ \cite[Sect.~4.2]{harrach2016enhancing}) that
\[
c_k=-\frac{1}{\lambda_\text{min} (L^{-1} \Lambda'(\sigma_0)\chi_{B_k} L^{-*})}
\]
where $\lambda_\text{min}(\cdot)$ denotes the smallest (most negative) eigenvalue, and $\Lambda(\sigma_0)-\Lambda(\sigma)=L L^*$ is the Cholesky factorization.

Let $g_1,\ldots,g_m\in L^2_\Diamond(\partial \Omega)$ denote $n$ orthonormal boundary currents, and $R(\kappa)$ be the Galerkin projection of 
the linearized data fitting residual to the span of these currents, i.e.,
\[
R:\ L^\infty(\Omega)\to \R^{m\times m}, \quad R(\kappa):=\left( \langle r(\kappa) g_i, g_j\rangle \right)_{i,j=1}^m.
\]

Then we have the following result for noiseless data.
\begin{theorem}
The quadratic minimization problem under box constraints
\[
\Vert{R(\kappa)}\Vert_F^2\to \text{min!}\quad \text{s.t.} \quad \kappa\in \mathcal P
\]
possesses a unique solution $\hat \kappa\in \mathcal P$. The support of $\hat \kappa$ agrees with the shape $D$ up to the utilized partition of $\Omega$, i.e.
\[
B_k\subseteq \mathrm{supp} \hat\kappa \quad \text{ if and only if } \quad B_k\subseteq D.
\]
\end{theorem}
\begin{proof}
\cite[Thm.~3.2]{harrach2016enhancing})
\end{proof}

For symmetric noisy data $\Lambda^\delta\in \mathcal L(L^2_{\Diamond}(\partial\Omega))$ with 
\[
\Vert \Lambda^\delta - \left(\Lambda(\sigma)-\Lambda(\sigma_0)\right) \Vert_{\mathcal{L}(L^2_{\diamond}(\partial\Omega))} < \delta,
\]
we define $R^\delta(\kappa):=\left( \langle r^\delta(\kappa) g_i, g_j\rangle \right)_{i,j=1}^m$ with 
$r^\delta(\kappa):  =  \Lambda^\delta - \Lambda^{\prime}(\sigma_0)\kappa$, and
\[
\mathcal P^\delta:=\left\{  \kappa\in L^\infty(\Omega):  \kappa=\sum_{k=1}^n a_k\chi_{B_k},\quad   0\leq a_k\leq  \min\{ \overline c,c^\delta_k \}\right\},
\]
with  
\[
c^\delta_k:=\max \{ \alpha>0:\  -\alpha \Lambda'(\sigma_0) \chi_{B_k}\preceq |\Lambda^\delta| + \delta I \}= -\frac{1}{\lambda_\text{min} (L_\delta^{-1} \Lambda'(\sigma_0)\chi_{B_k} L_\delta^{-*})},
\]
and the Cholesky factorization $|\Lambda^\delta|+\delta I=L_\delta L_\delta^*$. Then we have the following result.

\begin{theorem}
The quadratic minimization problem under box constraints
\[
\Vert{R^\delta(\kappa)}\Vert_F^2\to \text{min!}\quad \text{s.t.} \quad \kappa\in \mathcal P^\delta
\]
possesses a minimizer $\kappa\delta\in \mathcal P^\delta$. For each sequence of minimizers $\kappa^\delta$ with $\delta\to 0$ it holds that
$\kappa^\delta\to \hat \kappa$.
\end{theorem}
\begin{proof}
\cite[Thm.~3.8]{harrach2016enhancing})
\end{proof}

\subsection{Numerical results with the monotonicity method}\label{sec_linear}

For the following numerical examples, we use a FEM mesh of the geometry $\Omega=[0,1]\times [0,1]$ with $5248$ elements, $\sigma_1=2$, and $\sigma_0=1$. We consider the Neumann boundary data given by:
\begin{equation}\label{data}
g_k(x,y)= \sin(k\pi y)\chi_{\{x=0\}}- \sin(k\pi y)\chi_{\{x=1\}}+  \cos(k\pi x)\chi_{\{y=0\}}- \cos(k\pi x)\chi_{\{y=1\}},
\end{equation}
where $k=1,\ldots,m$, $m=23$, and $\chi_{\{x=a\}}$ is the indicator function on the boundary $x=a$.
We denote  by $\bar A\in \R^{m\times m}$ the discrete version of an operator $A\in \mathcal L (L^2_\Diamond(\partial \Omega))$, i.e., its Galerkin projection to the span of $g_1,\ldots,g_m$.

We used $100$ test balls. Following Theorem \ref{thm:mono_test_linearized}, for each test ball $B$, we computed the eigenvalues of
\[
\overline \Lambda(\sigma_0)- \overline\Lambda(\sigma)+\alpha \overline{\Lambda'}(\sigma_0)\chi_B 
\]
with $\alpha=1/2$. A test ball is marked as being inside the inclusion if all of these eigenvalues are positive.

We also considered the noisy data case 
\[
\overline\Lambda^\delta:=\overline \Lambda(\sigma)- \overline\Lambda(\sigma_0) + \delta \Vert \overline \Lambda(\sigma)- \overline\Lambda(\sigma_0)\Vert_F
\frac{E}{\Vert E\Vert_F},
\]
where the entries of $E\in \R^{m\times m}$ are normally distributed random variables with a mean of zero and a standard deviation of one.  
In the noisy data case, we mark those test balls for which all eigenvalues of 
\[
\alpha \overline{\Lambda'}(\sigma_0)\chi_B -\overline\Lambda^\delta+ \delta I
\]
are positive.

 \begin{figure}[H]
\begin{center}
\subfloat{\includegraphics[scale=0.27]{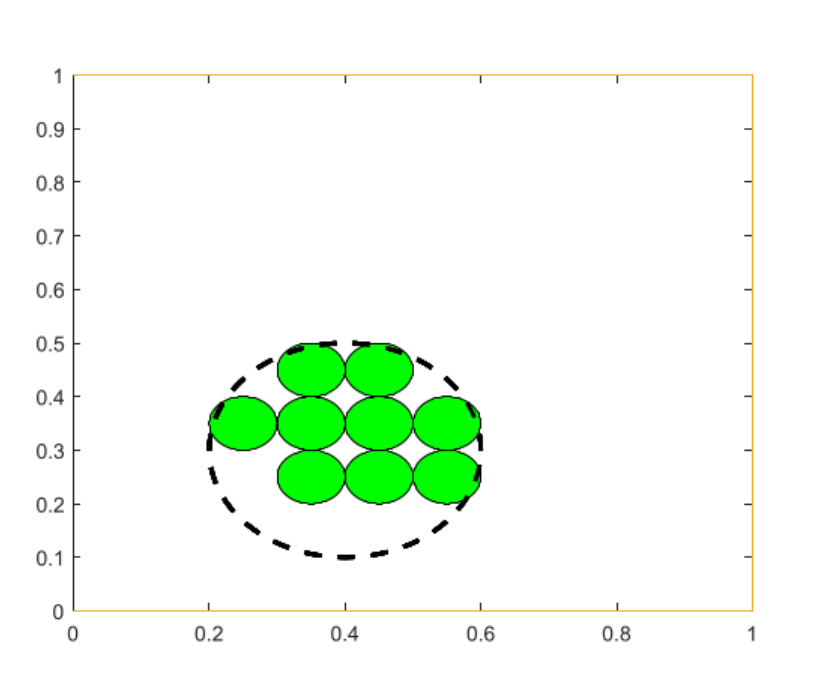}}\hfill
\subfloat{\includegraphics[scale=0.27]{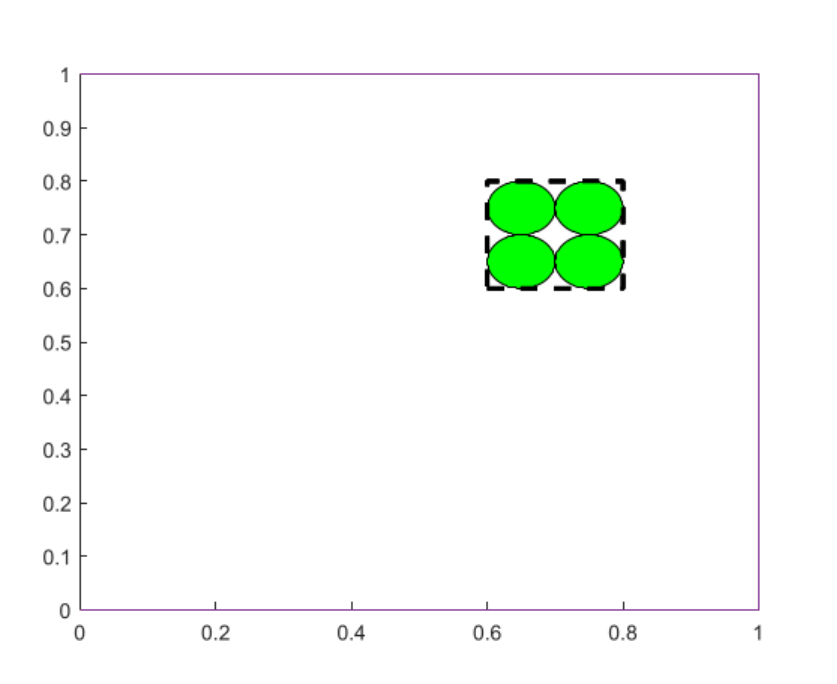}}\hfill
\subfloat{\includegraphics[scale=0.27]{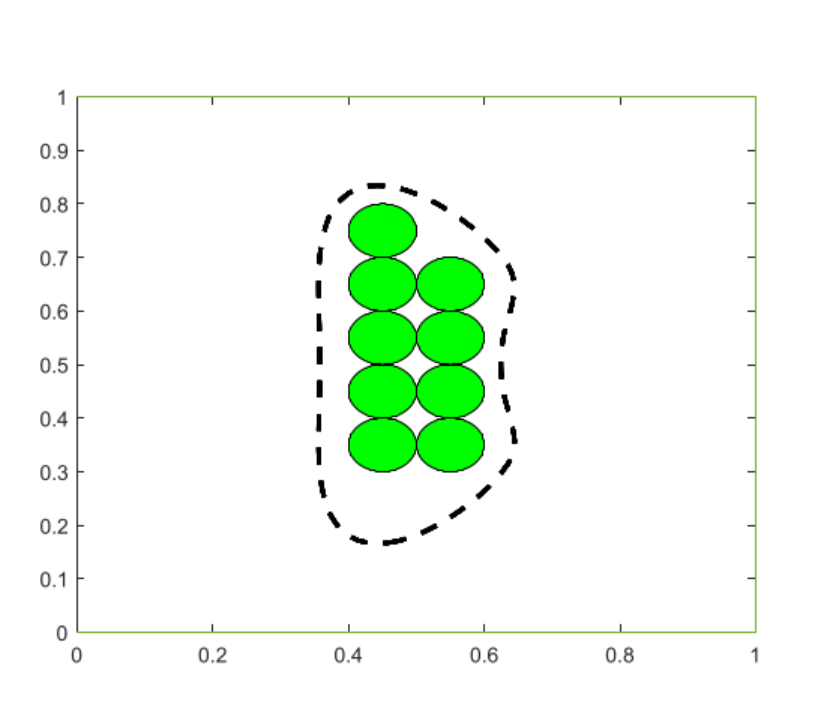}}\hfill
\subfloat{\includegraphics[scale=0.27]{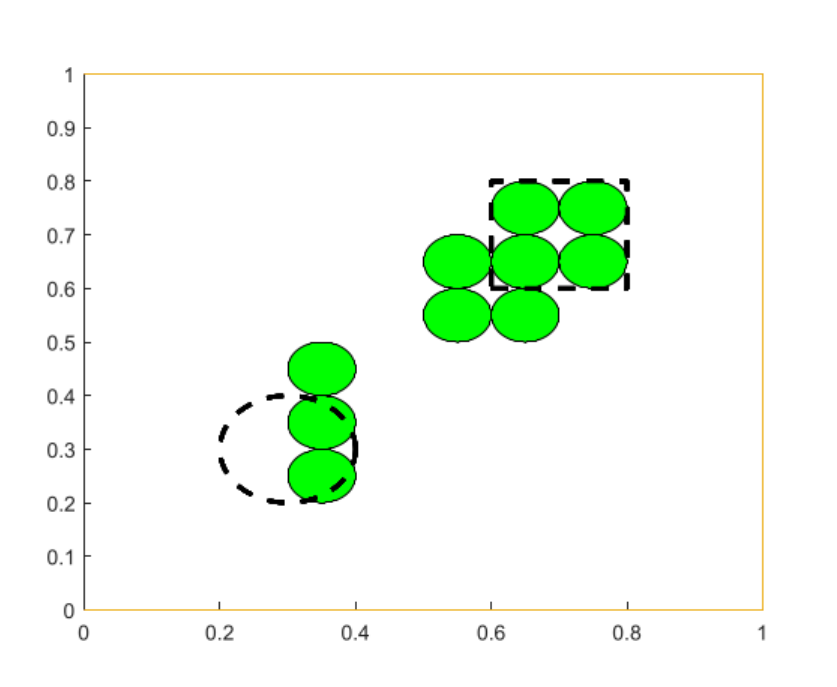}}
\end{center}
\caption{Reconstruction with the linearized monotonicity method in the case of noise-free data}\label{fig1}
\end{figure}
 \begin{figure}[H]
\begin{center}
\subfloat{\includegraphics[scale=0.27]{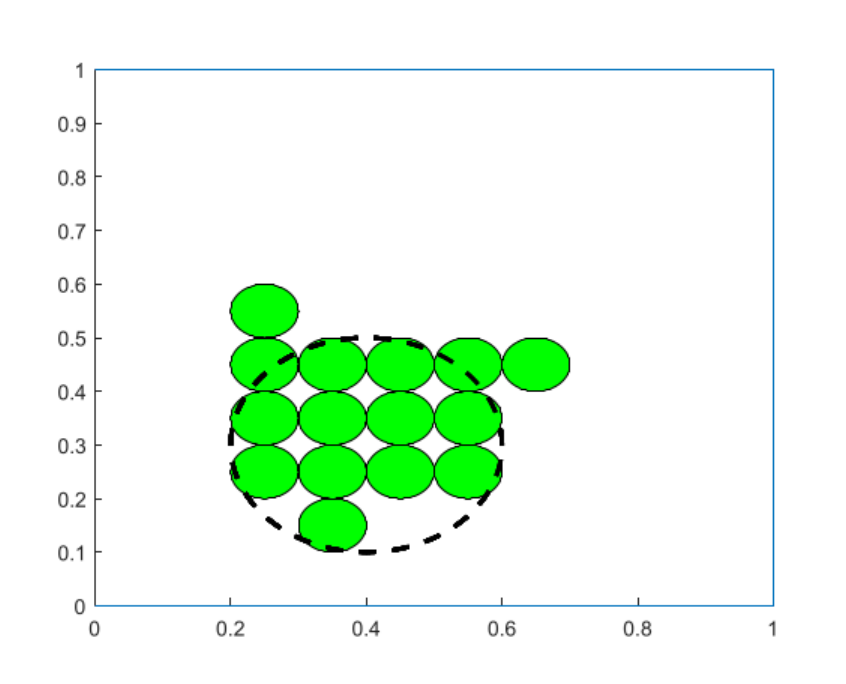}}\hfill
\subfloat{\includegraphics[scale=0.27]{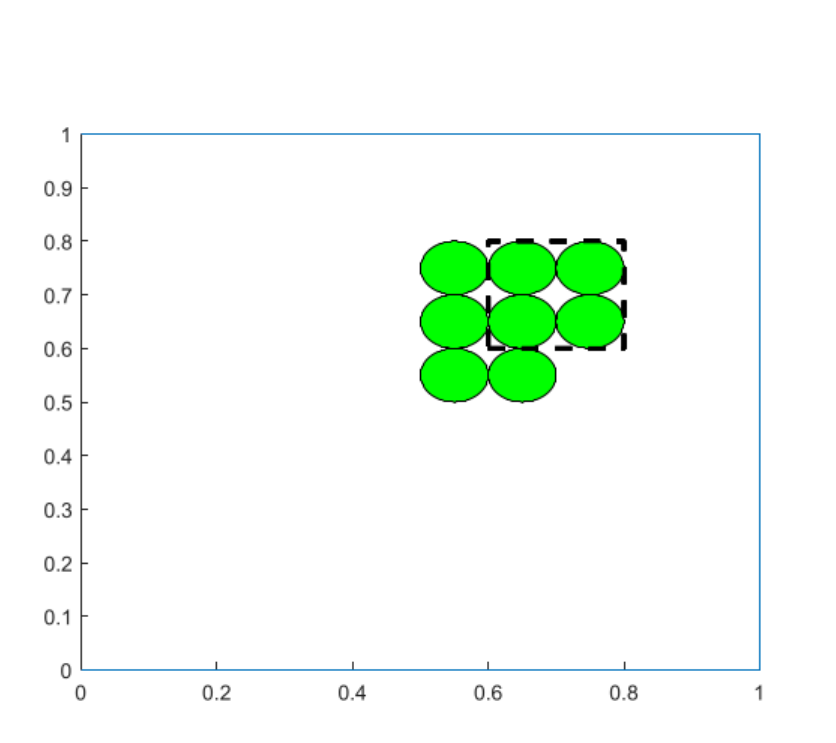}}\hfill
\subfloat{\includegraphics[scale=0.27]{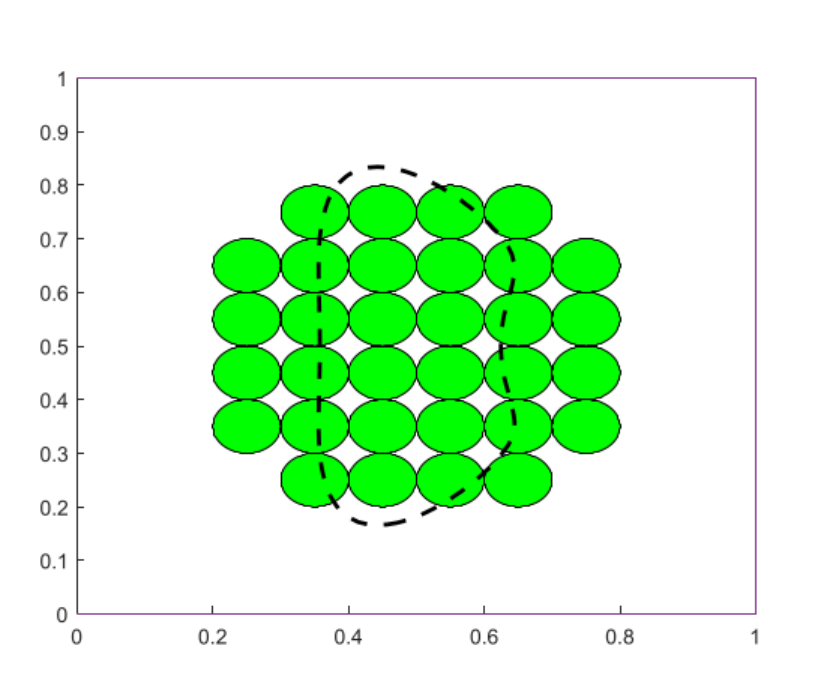}}\hfill
\subfloat{\includegraphics[scale=0.27]{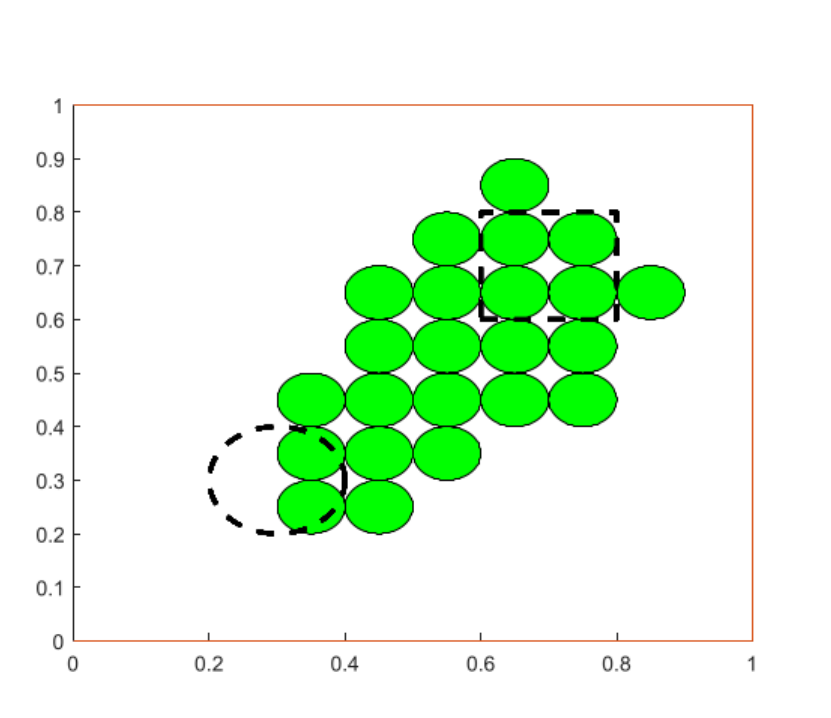}}
\end{center}
\caption{Reconstruction with the linearized monotonicity method in the case of noisy data, $\delta=0.1\%$}\label{fig2}
\end{figure}
Figures \ref{fig1} and \ref{fig2} show numerical results using the linearized monotonicity method for both noise-free and noisy data (with $\delta=0.001$). 
The numerical reconstructions comply with the theoretical results, but they also show high noise sensitivity.

We now turn to the monotonicity-based regularization method. We use the same geometry setting as before, but now use $n=100$ square pixels $B_k$ as a partition of $\Omega$. Denote $S_k:=\overline{\Lambda'}(\sigma_0)\chi_{B_k}\in \R^{m\times m}$, $k=1,\ldots,n$.
Following Section~\ref{subsect:mon_regularization}, for the noiseless case, we set
\[
c_k:=\max\{\alpha>0:\ \overline \Lambda(\sigma_0)+\alpha S_k \succeq \overline\Lambda(\sigma)\},
\]
$\overline c:=1/2$, and minimize
\[
\left\Vert \overline\Lambda(\sigma) - \overline \Lambda(\sigma_0) - \sum_{k=1}^n a_k S_k \right\Vert_{F}^2\to \mathrm{min!}\quad
s.t.\ \quad 0\leq a_k\leq \min\{ \bar c, c_k\}.
\]
In the noisy data case, we use
\[
c_k^\delta:=\max\{\alpha>0:\ -\alpha S_k \preceq |\overline \Lambda^\delta| +\delta I\},
\]
and minimize
\[
\left\Vert \overline \Lambda^\delta - \sum_{k=1}^n a_k S_k \right\Vert_{F}^2\to \mathrm{min!}\quad
s.t.\ \quad 0\leq a_k\leq \min\{ \bar c, c_k^\delta\}.
\]

\begin{figure}[H]
\begin{center}
\subfloat{\includegraphics[scale=0.27]{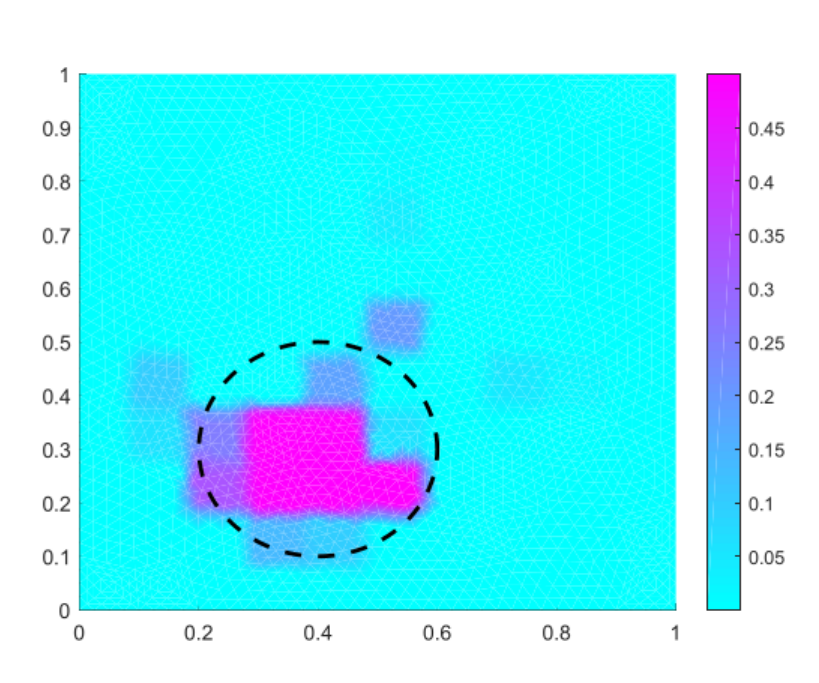}}\hfill
\subfloat{\includegraphics[scale=0.27]{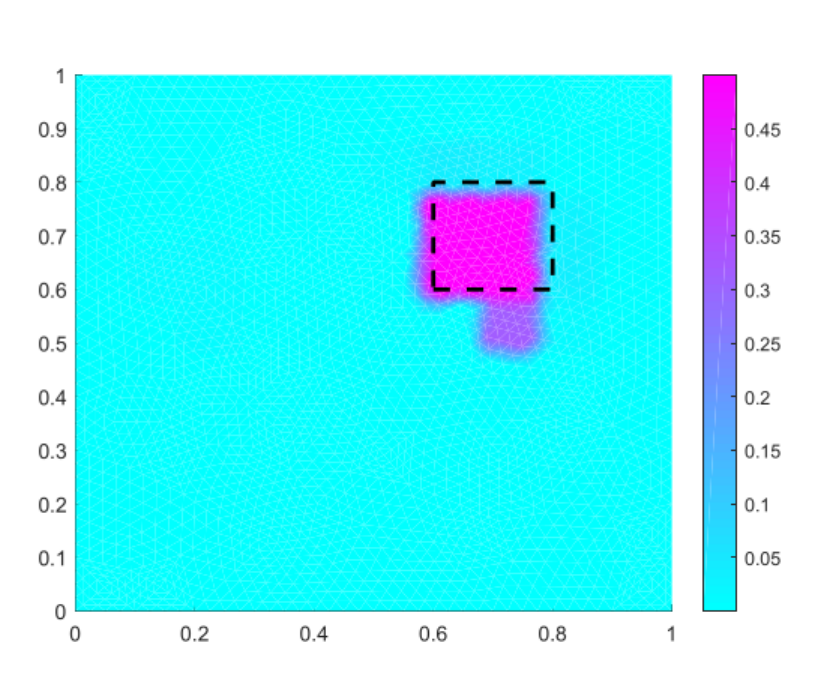}}\hfill
\subfloat{\includegraphics[scale=0.27]{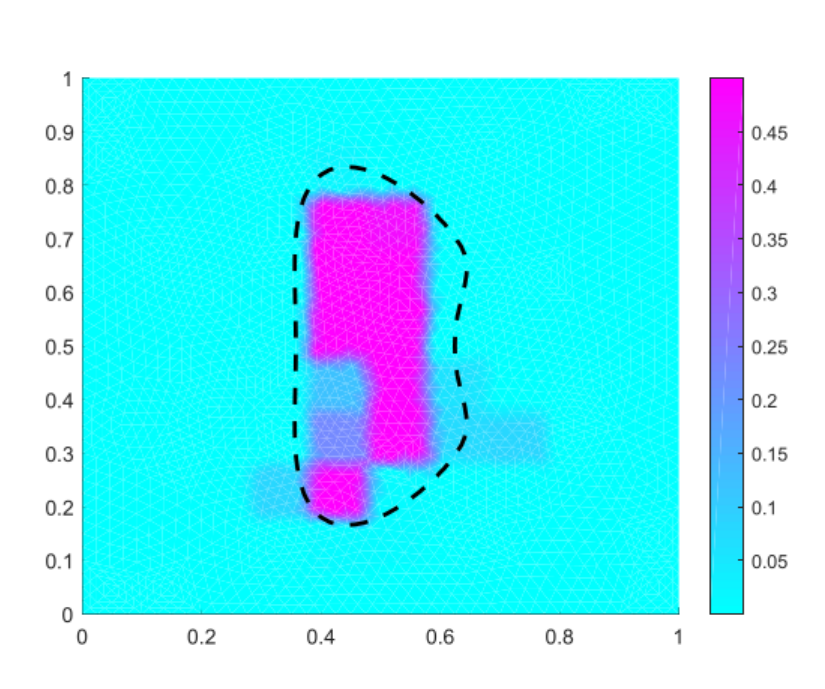}}\hfill
\subfloat{\includegraphics[scale=0.27]{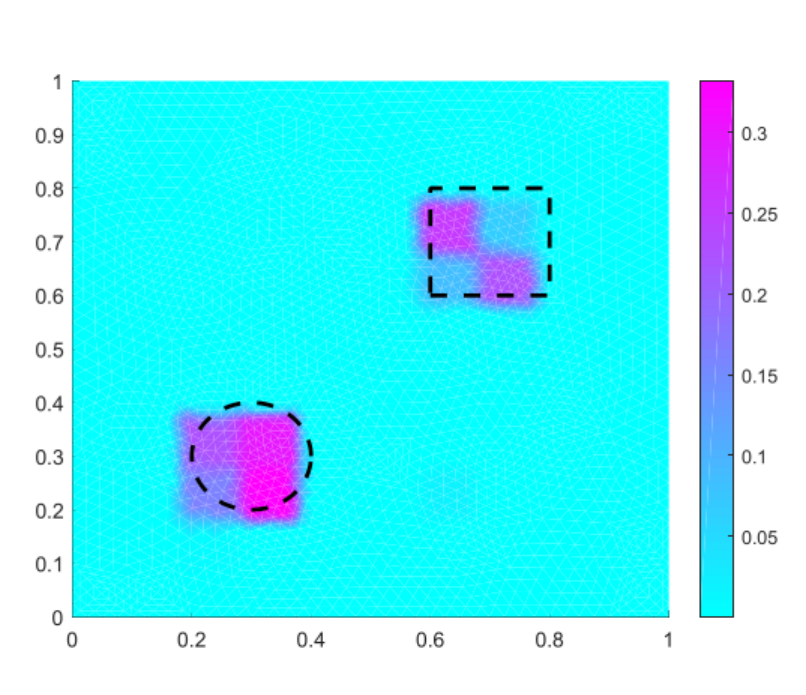}}
\end{center}
\caption{Reconstruction with the monotonicity-based regularization method in the case of noise-free data.}\label{fig_mon_reg_nonoise}
\end{figure}

 \begin{figure}[H]
\begin{center}
\subfloat{\includegraphics[scale=0.27]{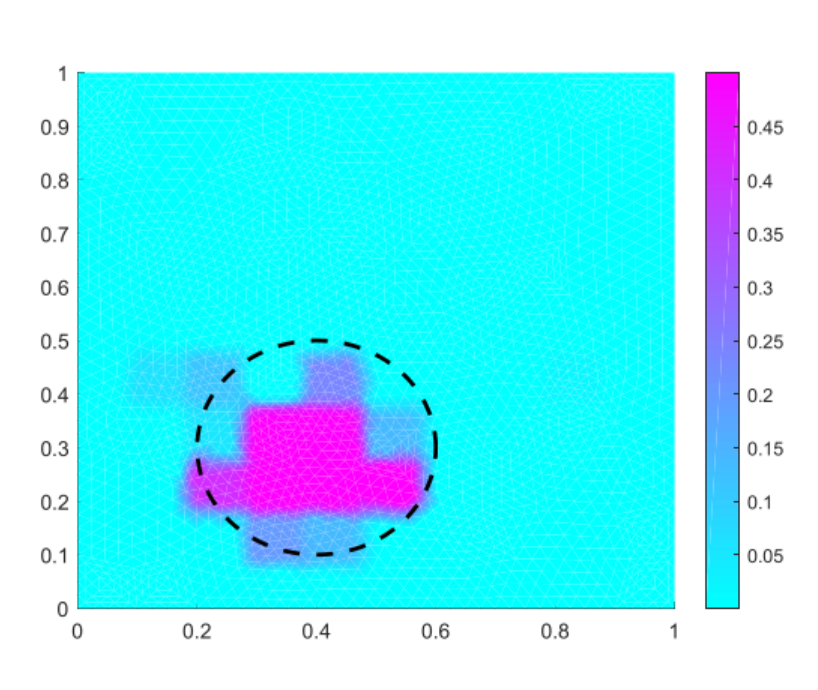}}\hfill
\subfloat{\includegraphics[scale=0.27]{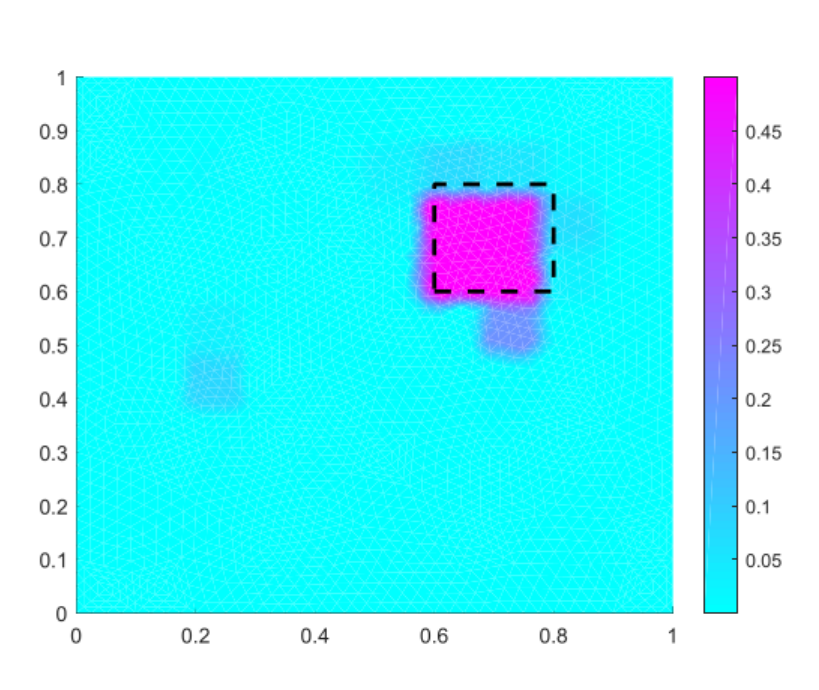}}\hfill
\subfloat{\includegraphics[scale=0.27]{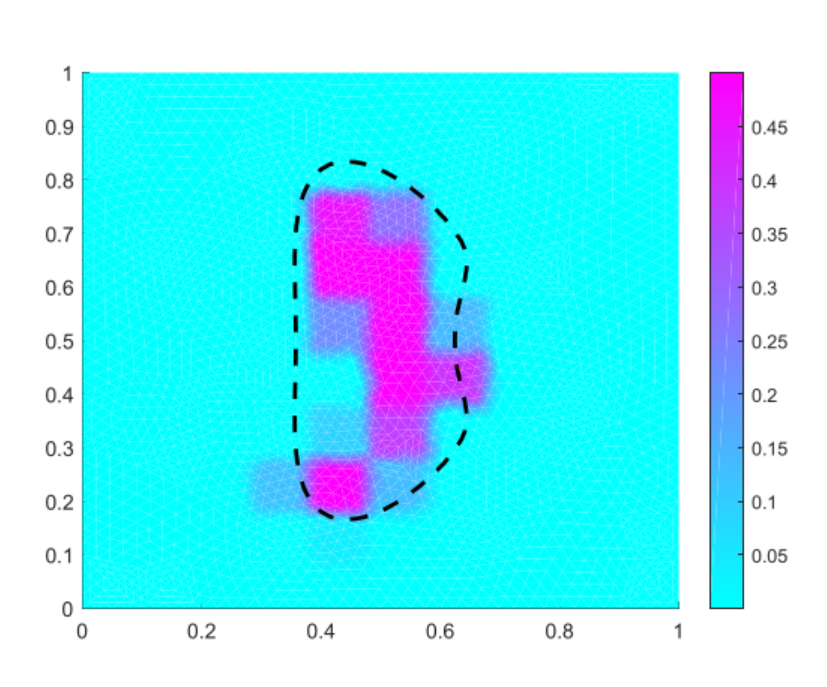}}\hfill
\subfloat{\includegraphics[scale=0.27]{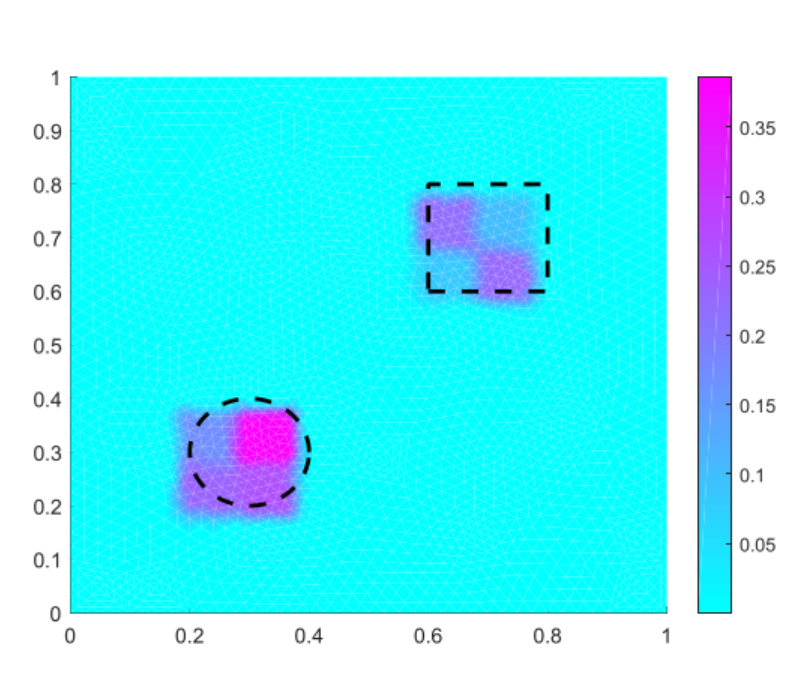}}
\end{center}
\caption{Reconstruction with the monotonicity-based regularization method in the case of noisy data, $\delta=10\%$.}\label{fig_mon_reg_noise}
\end{figure}

Figure~\ref{fig_mon_reg_nonoise} and \ref{fig_mon_reg_noise} show numerical results using the monotonicity-based regularization method for both noise-free and noisy data (with $\delta=0.1$). 
They clearly demonstrate the highly improved noise robustness for the monotonicity-based regularization method. 
The minimization was carried out using CVX, a package for specifying and solving convex programs \cite{GrBo07,cvx}. Note that, as it was already observed in \cite{harrach2016enhancing,harrach2018monotonicity}, the CVX package provides superior reconstructions compared to using Matlab's standard \texttt{quadprog} function. Moreover, in our numerical experiments, we have again observed that almost equally good results can be obtained with the simplified constrained $0\leq a_k\leq \bar c$, though a theoretical explanation for this is still missing.

\section{Shape recovery via  the Kohn-Vogelius functional}
In this section, we turn to an iterative strategy to recover the unknown shape $D$. 
To transforming the inverse problem \eqref{invp} into a minimization problem, we use a Kohn-Vogelius type functional to achieve shape recovery. 
This approach allows us to frame the problem as follows:
  \begin{equation}\label{min-J}
 \left\{\begin{aligned}
&\mbox{minimize } \J(D,u_{D},v_{D}):=\int_\Omega \sigma \vert \nabla(u_D-v_D) \vert^2\,dx \\
&\text{subject to } D\subset\mathcal{O}_{ad},  u_{D} \text{ and } v_{D} \text{ solutions of } \eqref{EIT}  \text{ and } \eqref{dir}, 
\end{aligned}\right.
\end{equation}
 where  $u_{D}$ is the solution of the Neumann problem \eqref{EIT}
and $v_{D}$  is the solution of the Dirichlet problem
\begin{equation}\label{dir}
\left\{
\begin{aligned}
   -\text{div}(\sigma \nabla v_{D}) & = 0 \quad \text{in }\Omega,
\\
  v_{D}&= f \quad \text{ on  }\partial \Omega,\\
\end{aligned}
\right.
\end{equation}
where $f\in H^{1/2}(\partial\Omega)$ is a measurement of the potential corresponding to the input flux $g$.

To solve numerically  the minimzation problem \eqref{min-J}, we need to compute the shape derivative of the Khon-Vogelius functional $\J$.

First, we recall some basic facts  about  the velocity method from shape optimization used  to calculate the shape derivatives of the functionals $\J$, see \cite{DZ2,SZ}. 
In the velocity (or speed) method a domain $\Omega$ is deformed by  the action of  a velocity field $V$. The evolution of the domain is described by the following  dynamical system:
\begin{equation}
 \label{speed}
  \left\{
 \begin{aligned}
 \frac{d}{dt}x(t)&=V(x(t)),  t\in [0,\varepsilon )\\
 x(0)&=X
\end{aligned}\right.
\end{equation} 
for some real number $\varepsilon>0$.   Assume $V \in \mathcal{D}^1(\Omega;\R^2)$ where  $ \mathcal{D}^1(\Omega;\R^2)$    denotes the space of continuously differentiable functions with compact support  in $\Omega$, then the ordinary differential equation \eqref{speed}  has a unique solution. 
This allows us to define the diffeomorphism
\[
\label{transf}
T_{t} :\R^2\rightarrow \R^2 : X\mapsto T_{t}(X):=x(t), 
\]
For $t\in [0,\varepsilon)$, $ T_t$ is invertible  satisfies  $T_t(\Omega)=\Omega$  but $T_t(D)\neq D$.
The Jacobian $\xi(t)$   defined by 
\[
\label{Jt}
\forall\; t\in[0,\varepsilon),\quad \xi(t)=  |\text{det}DT_{t}(X)|>0,
\]
where $DT_{t}(X)$ is the Jacobian matrix of  the transformation $T_{t}$ associated with the velocity field $V$, is  derivable with respect to $t$
and 
\begin{equation}
\xi^{\prime}(0)=  I_2- DV-DV^T.
\end{equation}

%%%%%%%%%%%%%%%%%%%%%%%%%%%%%%%%%%
Let $J$ be a real valued function $J: \Omega \longrightarrow \R$. We say that $J$ has a {\it Eulerian semiderivative} at $\Omega$ in the direction $V$ if the following limit
exists and is finite:
\[
dJ(\Omega;V)= \lim_{t\searrow 0}\frac{J(T_{t}(\Omega)) - J(\Omega)}{t}.
\]
If the map $V\longrightarrow dJ(\Omega;V))$ is linear and continuous, we say that $J$ is shape differentiable at $\Omega$.

Now,  we  state the shape derivative of the functional  $\J$ with respect to the shape $D$. We introduce the reduced functional
$J(D):= \J(D, u_D,v_D)$.
\begin{theorem}\label{shape_deriv}
The functional  $J$ is  shape  differentiable with  respect to $D$ and its derivative in the direction  $V$ is given by 
\begin{equation}\label{shape_der}
dJ(D;V)=\int_\Omega \left( \xi^{\prime}(0)\nabla v_D\cdot \nabla v_D -\xi^{\prime}(0)\nabla u_D\cdot \nabla u_D\right)\,dx. 
\end{equation}
\end{theorem}
\begin{proof}
The proof of Theorem \ref{shape_deriv}  can be found in \cite{belhachmi2023level,laurain2016shape}.
\end{proof}
%%%%%%%%%%%%%%%%%%%%%%%%%%%%%%%%%
\subsection{Algorithm and numerical results}\label{numerics}
In this subsection, we use the same conductivity values as in Section 3 and apply specific boundary fluxes,
 which are derived from the data in \eqref{data} and presented as follows:
\begin{align*}
%g_1&= 1\chi_{\{x=0\}}+ 1\chi_{\{x=1\}} - 1\chi_{\{y=0\}}- 1\chi_{\{y=1\}},\\
%g_2&= 1\chi_{\{x=0\}}+ 1\chi_{\{y=1\}} - 1\chi_{\{x=1\}}- 1\chi_{\{y=0\}},\\
%g_3&= 1\chi_{\{x=0\}}+ 1\chi_{\{y=0\}} - 1\chi_{\{x=1\}}- 1\chi_{\{y=1\}},\\
%g_4&= 0\chi_{\{x=0\}}+ 1\chi_{\{y=0\}} - 1\chi_{\{x=1\}}- 0\chi_{\{y=1\}},\\
%g_5&= 0\chi_{\{x=0\}}+ 0\chi_{\{y=0\}} + 1\chi_{\{x=1\}}- 1\chi_{\{y=1\}}.
g_1&= \frac{g_k(0, y) - g_k(1, y)}{\sin(k\pi y)} - \frac{g_k(x, 0) +g_k(x, 1)}{\cos(k\pi x)},\\
g_2 &= \frac{g_k(0, y) - g_k(1, y)}{\sin(k\pi y)} - \frac{g_k(x, 1) - g_k(x, 0)}{\cos(k\pi x)},\\
g_3 &= \frac{g_k(0, y) - g_k(1, y)}{\sin(k\pi y)} + \frac{g_k(x, 0) - g_k(x, 1)}{\cos(k\pi x)},\\
g_4& = \frac{g_k(x, 0)}{\cos(k\pi x)} - \frac{g_k(1, y)}{\sin(k\pi y)},\\
g_5 &= \frac{g_k(1, y)}{\sin(k\pi y)} - \frac{g_k(x, 1)}{\cos(k\pi x)}.
\end{align*}
To simulate  noisy  data, the measurements $f_k$   ($f_k=u_k\vert_{\partial \Omega}$  where $u_k$ is solution of the direct problem with
 $g=g_k,  k=1,\ldots,3$) are corrupted by adding a normal Gaussian noise with mean zero and standard deviation $\eta*\|f_k\|_\infty$, where $\eta$ is a parameter.

 For the numerical implementation,  we use the software package FEniCS; see \cite{fenics:book}. 
  The avoid the so-called  inverse crime problem, the domain $\Omega$ is meshed using crossed  grid of $128\times 128$ elements to compute the measurements $f$.
Then we use  a crossed grid  of  $64\times 64$ to solve the minimization problem. 

The evolution of the shape $D$ is modeled using the level set method originally introduced in \cite{MR965860}, which gives a general
 framework for the computation of evolving interfaces using an implicit representation of these interfaces.  
The core idea of this method is to represent $\Gamma_t$ as the level set of a continuous $\phi: \Omega\times \R^+ \to \R$, the so-called level set function:
\begin{equation}
D_t = \{x\in \Omega,\ \phi(x,t) < 0\}.
\end{equation}
Let $x(t)$ be the position of a particle on $D_t$ moving with velocity $\dot{x}(t)=V(x(t))$ according to \eqref{speed}.
Differentiating the relation $\phi(x(t),t)=0$ with respect to $t$
yields the Hamilton-Jacobi equation:
\begin{equation}
\partial_t\phi + V\cdot \nabla \phi = 0  \quad \mbox{ in } D_t\times\R^+,
\label{eq:transport}
\end{equation}
which can be extended to $\Omega\times\R^+$.
For the discretization of \eqref{eq:transport} we use the Local Lax-Friedrichs flux from \cite{MR1111446} and a forward Euler time discretization.
The shape gradients of $J$  is computed in the $H^1$-norm, i.e. we solve for instance the equation 
\begin{equation}\label{smoothing_num}
\mbox{ Find }V\in H^1_0(\Omega) : \int_\Omega D V : D W = - dJ(D; W)\mbox{ for all } W\in H^1_0(\Omega) 
\end{equation}
where $dJ(D; W)$ is given by \eqref{shape_der}.\\
%--------------------------------------------------------------------------------------------------------
The algorithm proceeds through the following steps
 \begin{itemize}
 \item[1.]  Initialize the level-set function $\phi_0$ based on the initial guess $D_0$. 
  \item[2.]   Repat  until convergence, for $k\geq 0$:
   \begin{itemize}
  \item[a.] Compute the solutions $u_k$  of \eqref{EIT},  associated  with  the shape
  \[D_k=\left\{x\in \Omega:  \phi_k(x) <0\right\}\].
     \item[b.]  Deform the shape $D_k$   by transporting  o the level set function
$\phi_{k+1}(x)=\phi(x,\Delta t_k)$,  where $\phi(x,t)$  is   solution of  \eqref{eq:transport}
with velocity $V$  given by  \eqref{smoothing_num}   and initial condition 
 $\phi(x,0)=   \phi_k(x)$.
   The time step is selected so that $J(D_{k+1})< J(D_{k})$. 
   \end{itemize}
 \end{itemize}
%----------------------------------------------------------------------------------------------------------%
 \begin{figure}[H]
\begin{center}
\subfloat{\includegraphics[scale
=0.27]{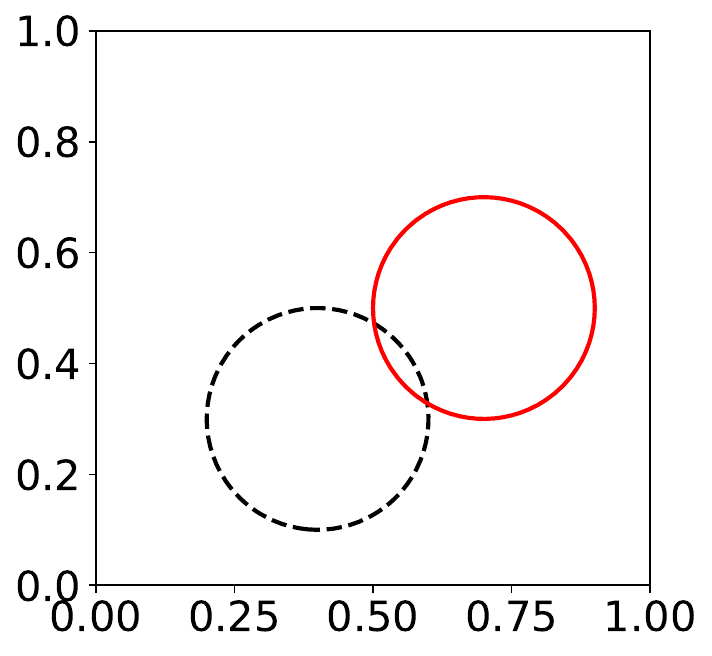}}\hfill
\subfloat{\includegraphics[scale=0.27]{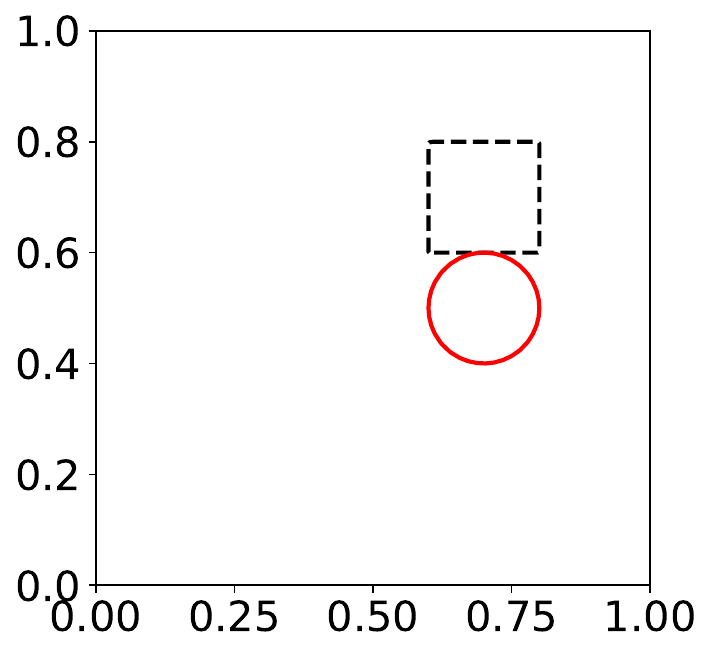}}\hfill
\subfloat{\includegraphics[scale=0.27]{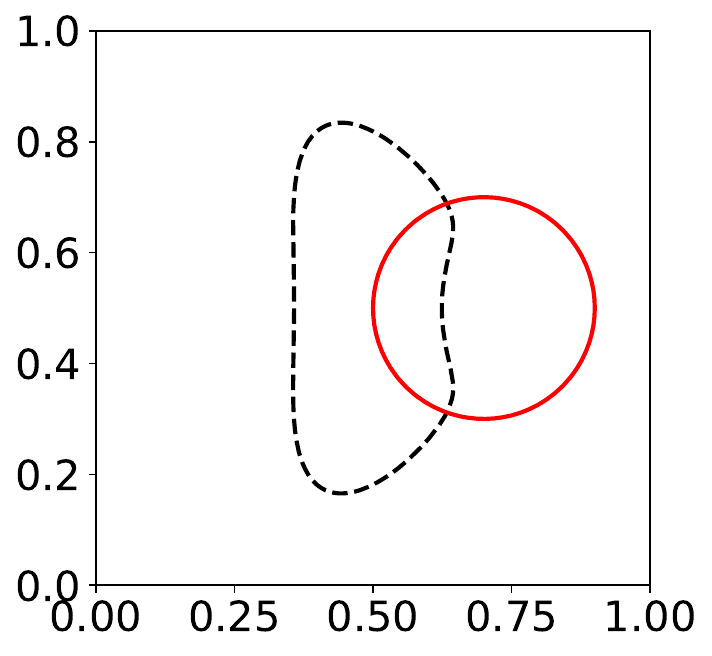}}\hfill
\subfloat{\includegraphics[scale=0.27]{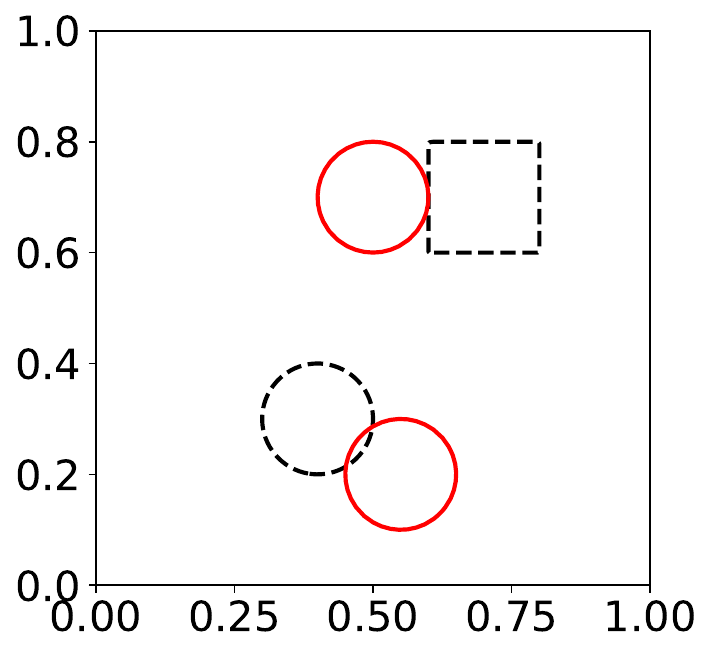}}\\
\subfloat{\includegraphics[scale=0.27]{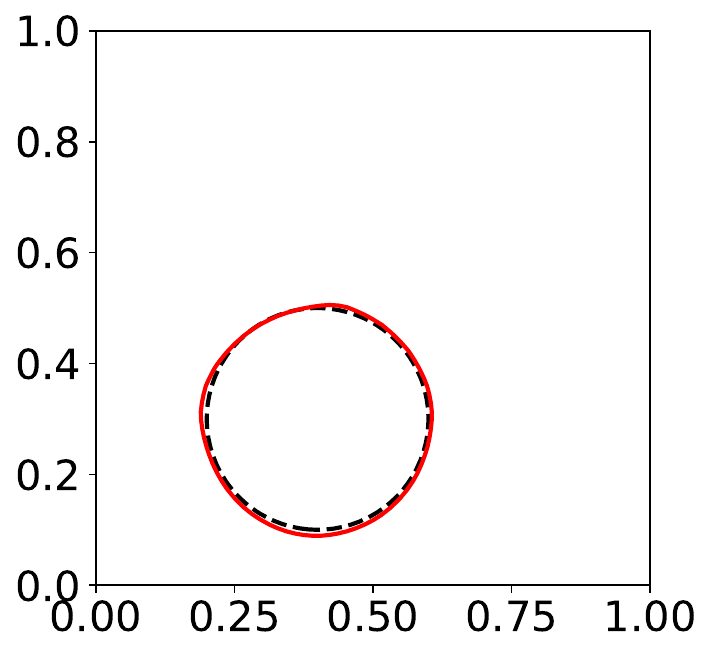}}\hfill
\subfloat{\includegraphics[scale=0.27]{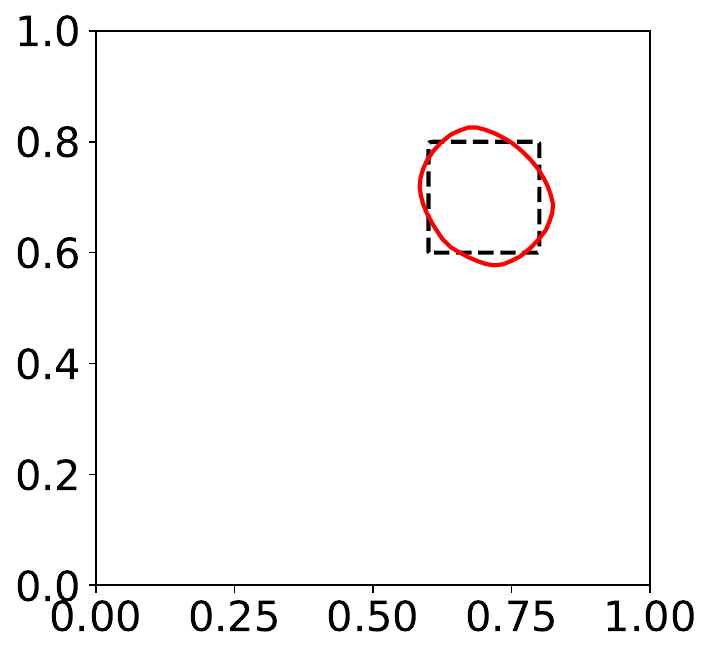}}\hfill
\subfloat{\includegraphics[scale=0.27]{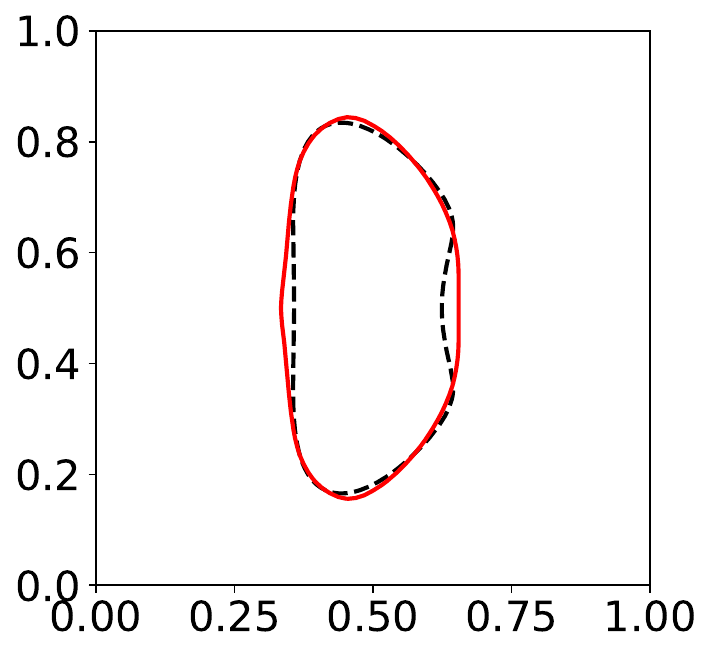}}\hfill
\subfloat{\includegraphics[scale=0.27]{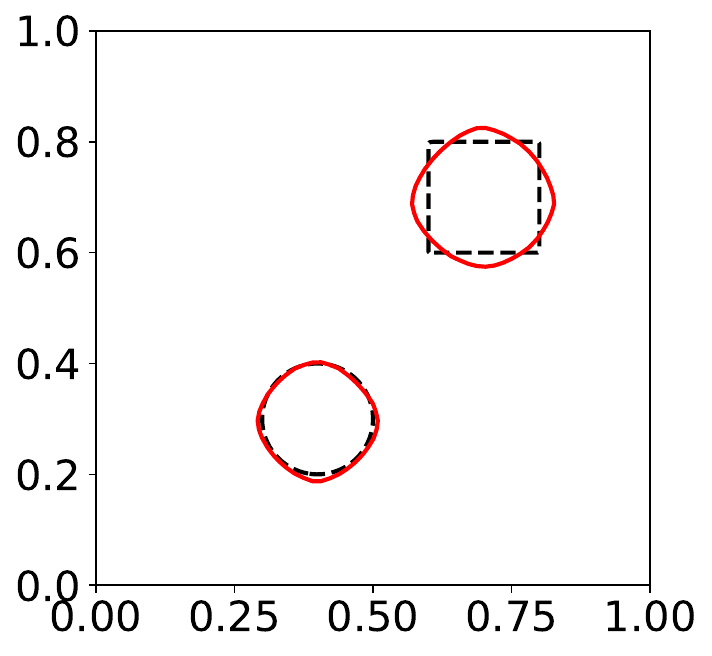}}\\
\subfloat{\includegraphics[scale=0.27]{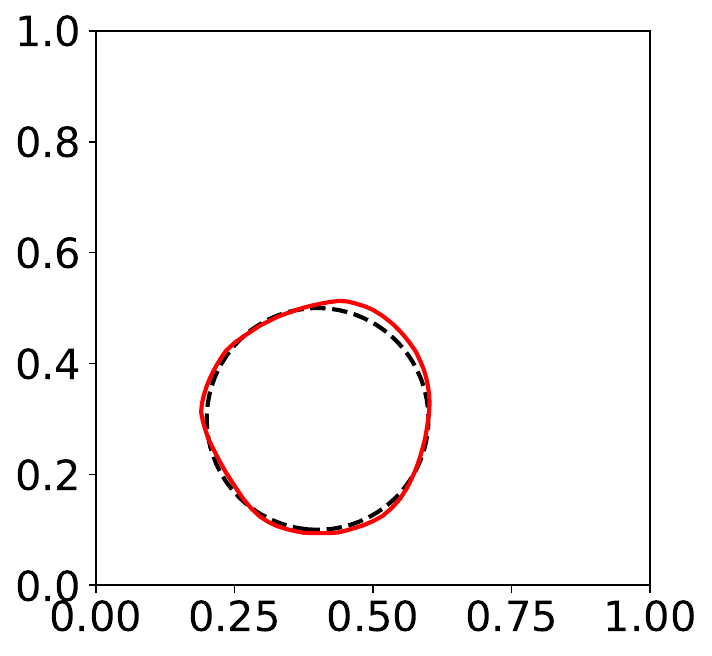}}\hfill
\subfloat{\includegraphics[scale=0.27]{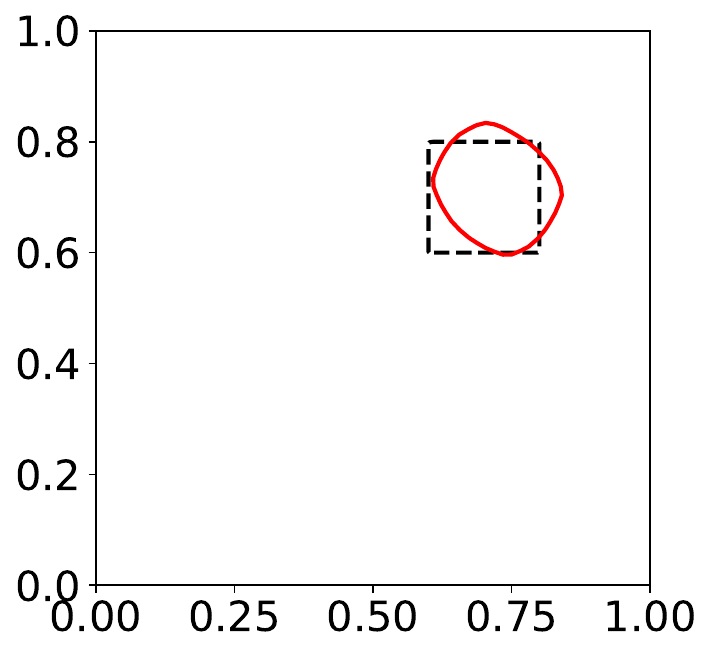}}\hfill
\subfloat{\includegraphics[scale=0.27]{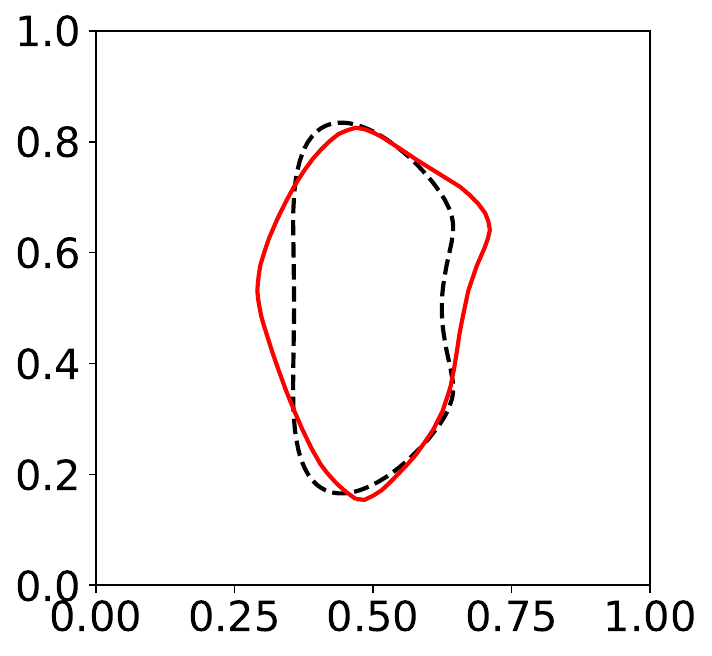}}\hfill
\subfloat{\includegraphics[scale=0.27]{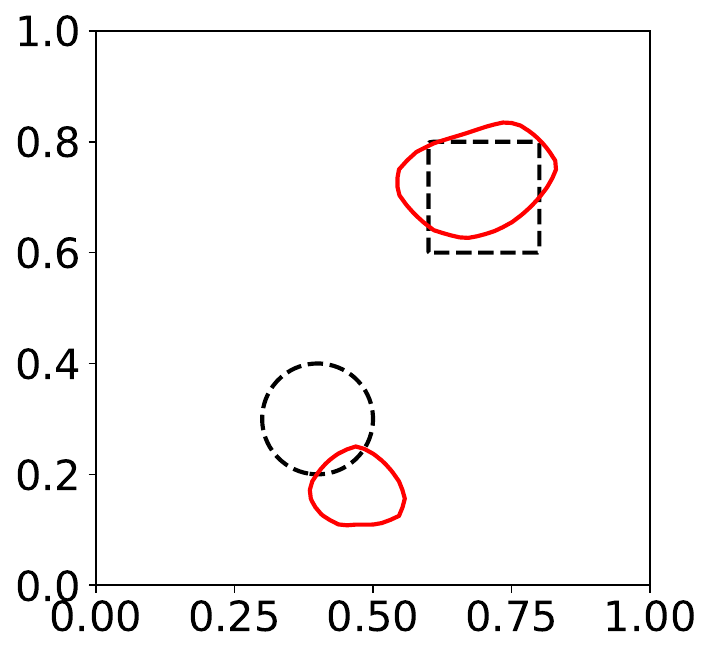}}\\
\subfloat{\includegraphics[scale=0.27]{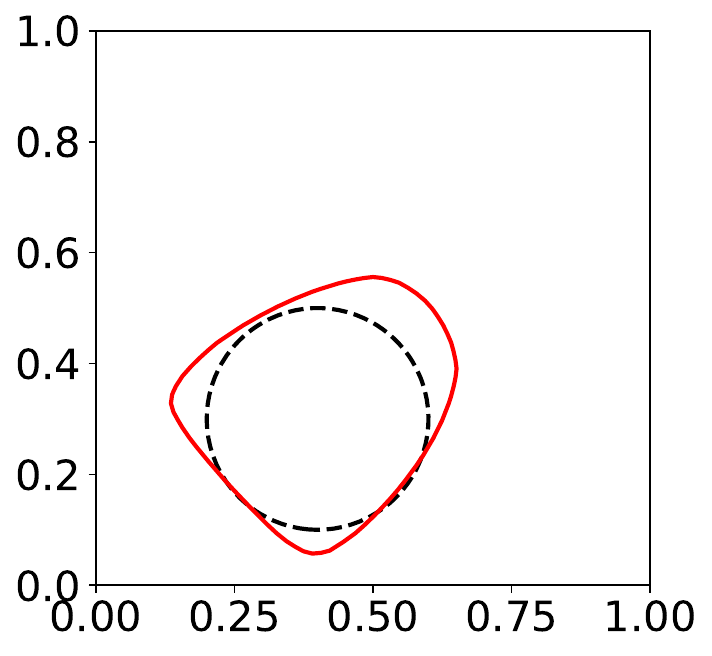}}\hfill
\subfloat{\includegraphics[scale=0.27]{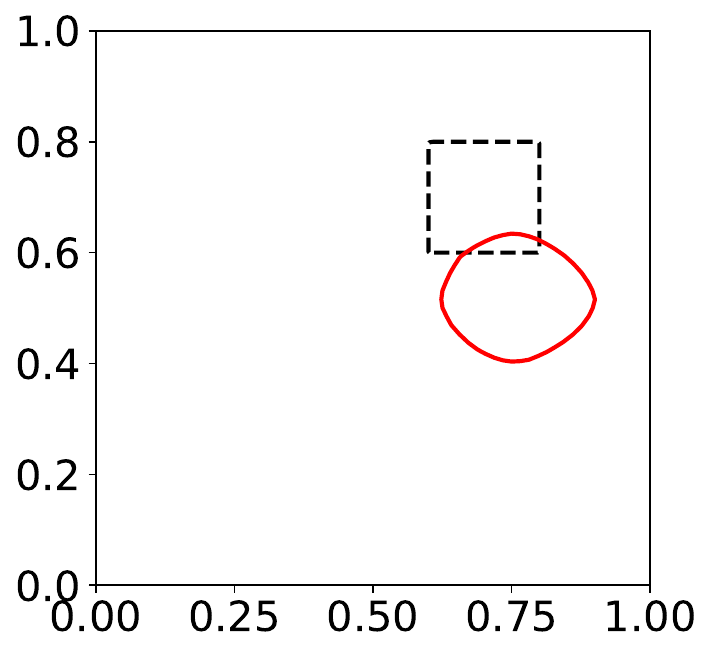}}\hfill
\subfloat{\includegraphics[scale=0.27]{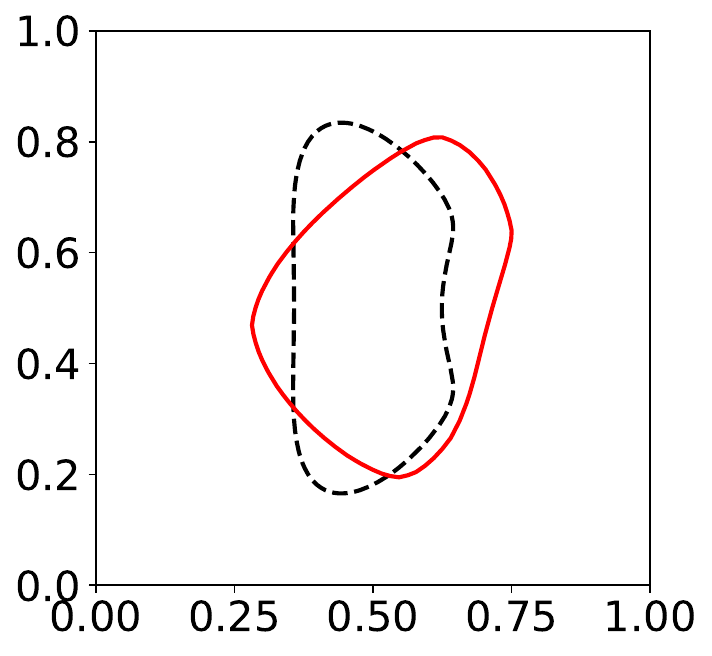}}\hfill
\subfloat{\includegraphics[scale=0.27]{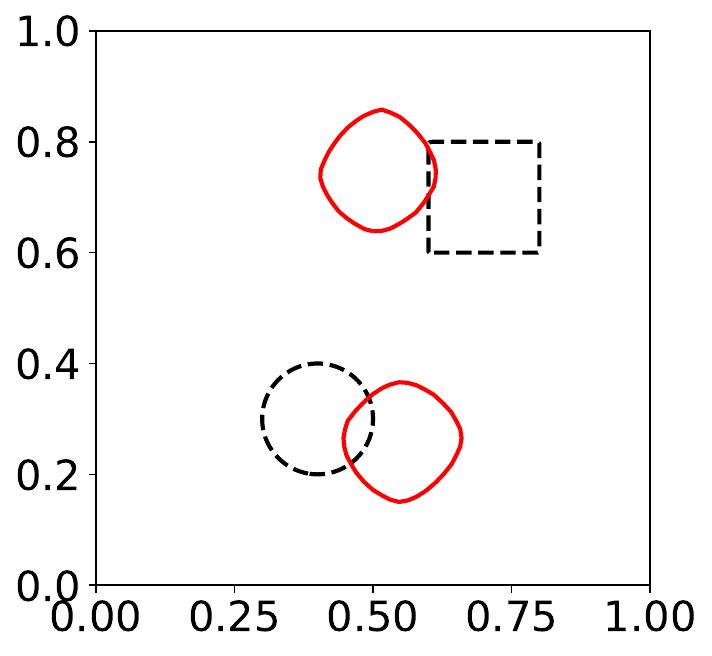}}
\end{center}
\caption{Reconstruction  with the levelset  method. First row: initialization (solid contours) and true inclusion (dashed contour).
 Second row: reconstruction with free noise  (solid contours) and true inclusion (dashed contours). Convergence occurs in 70, 79, 87 and 60 iterations (from left to right).
 Third row: reconstruction with  noise level $\eta =0.01$  (solid contours) and true inclusion (dashed contours). Convergence occurs in 78, 82, 125 and 110 iterations (from left to right).
 Fourth row: reconstruction with noise level   $\eta =0.03$ (solid contours) and true inclusion (dashed contours). Convergence occurs in 105, 90, 175 and 120 iterations (from left to right).
}
\label{fig_levelset}
\end{figure}
\noindent
Figure \ref{square_bis} presents a numerical example demonstrating the failure of the level set method to converge with an alternative
 initialization, in contrast to the initialization  given  by the monotonicity method (see, second column of Figure \ref{fig_combine}). This motivates us to employ the 
monotonicity method  to achieve a robust initialization and address the challenges posed by improper initialization
 in the level set method.
\begin{figure}[H]
\begin{center}
\subfloat{\includegraphics[scale=0.45]{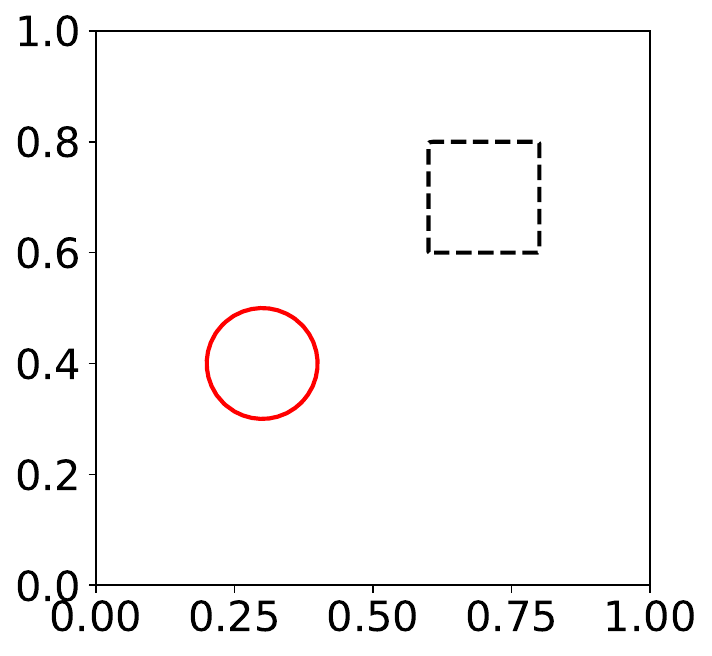}}
\subfloat{\includegraphics[scale=0.45]{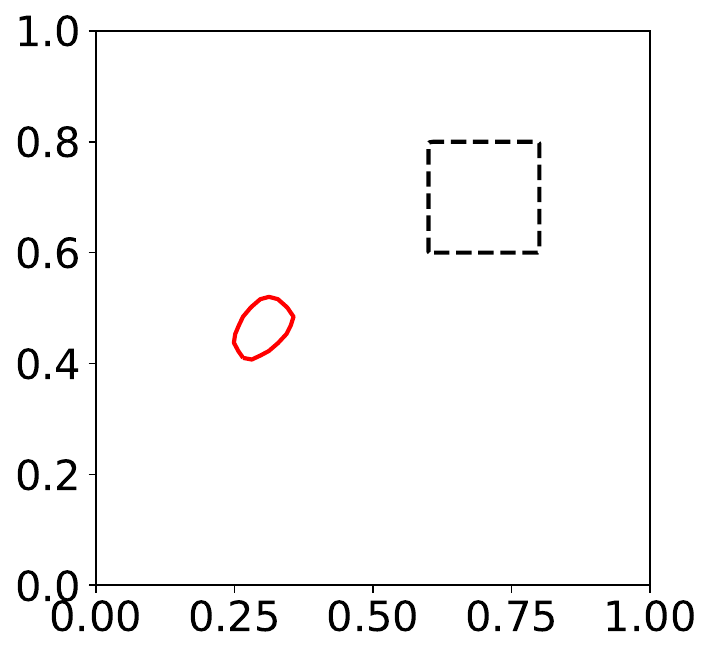}}
\end{center}
\caption{ Reconstruction  with the levelset  method: initialization (solid contours) and true inclusion (dashed contour). 
Right figure: reconstruction with  noise level $\eta =0.01$. The algorithm stop afer 72 iterations.}\label{square_bis}
\end{figure}
\begin{remark}
We emphasized that, for shape recovery using the Kohn-Vogelius process, the measurements \(\langle \Lambda(\sigma) - \Lambda(\sigma_0) g_i, g_j \rangle\), required in the regularized monotonicity method, are not essential. Instead, partial measurements of the form \((\Lambda(\sigma) g_i)\) are sufficient. To improve the computational efficiency of the level set method, a carefully selected subset of boundary data was extracted from the dataset specified in \eqref{data}.
\end{remark}
%----------------------------------------------------------------------------------------------%
\section{Combined   monotonicty and level-set method  for  shape reconstruction}
In this section, we present numerical results using a combination of   monotonicity and level set methods. More precisely,
 we begin by applying the monotonicity method to obtain an initial approximation of the solution. We then refine this approximation
 using the level set method for improved resolution.  

Since the regularized monotonicity method provides a better approximation compared to the linearized monotonicity method,
 especially in the presence of noisy data, we use the results shown in Figure \ref{fig_mon_reg_noise} as the initialization for the numerical results 
in Figure \ref{fig_combine}. For simplicity, we approximate the results in Figure \ref{fig_mon_reg_noise} using basic geometries, such as circles or ellipses.
%---------------------------------------------------------------%
 \begin{figure}[H]
\begin{center}

\subfloat{\includegraphics[scale=0.27]{results_cvx/circle_cvx_box01}}\hfill
\subfloat{\includegraphics[scale=0.27]{results_cvx/square_cvx_box01}}\hfill
\subfloat{\includegraphics[scale=0.27]{results_cvx/concave_cvx_box01}}\hfill
\subfloat{\includegraphics[scale=0.27]{results_cvx/circle_square_cvx_box01}}\\
%\subfloat{\includegraphics[scale=0.27]{combine_results/initial_circle_combine}}\hfill
%\subfloat{\includegraphics[scale=0.27]{combine_results/initial_square_combine}}\hfill
%\subfloat{\includegraphics[scale=0.27]{combine_results/initial_concave_combine}}\hfill
%\subfloat{\includegraphics[scale=0.27]{combine_results/initial_square_circle_combine}}\\
\subfloat{\includegraphics[scale=0.27]{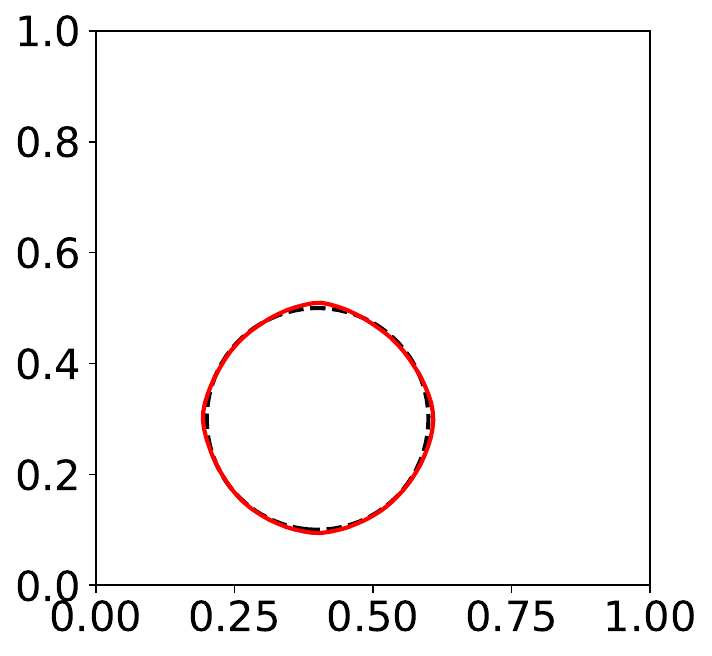}}\hfill
\subfloat{\includegraphics[scale=0.27]{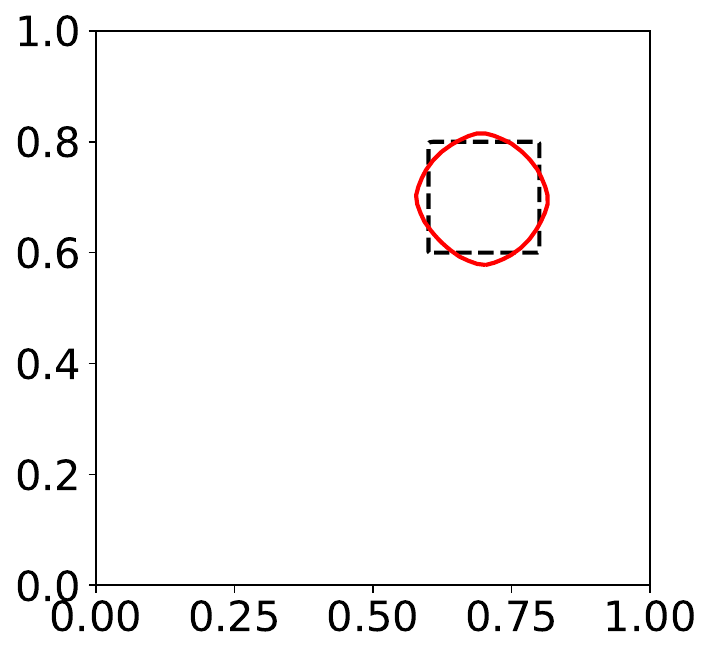}}\hfill
\subfloat{\includegraphics[scale=0.27]{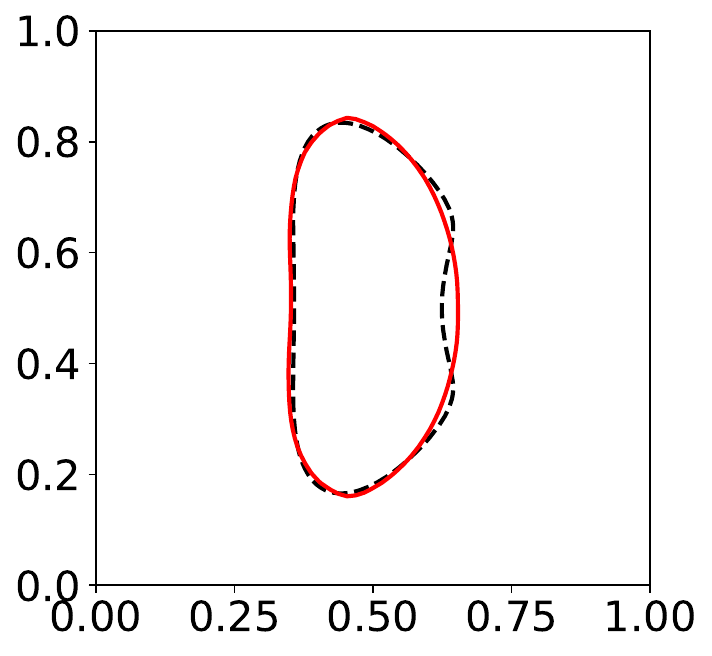}}\hfill
\subfloat{\includegraphics[scale=0.27]{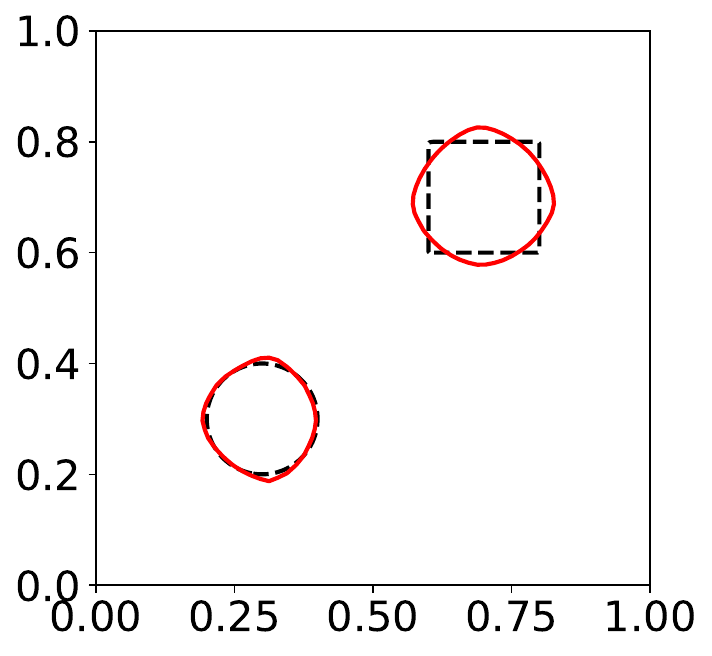}}\\
\subfloat{\includegraphics[scale=0.27]{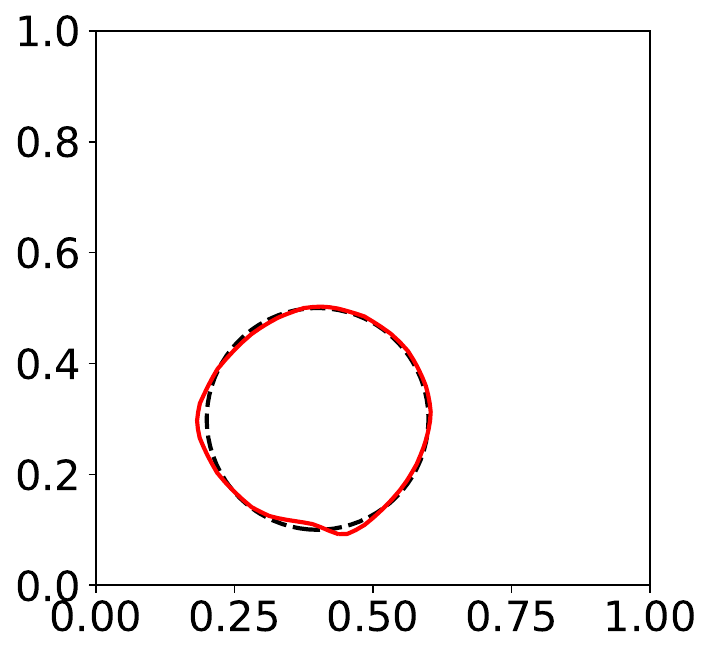}}\hfill
\subfloat{\includegraphics[scale=0.27]{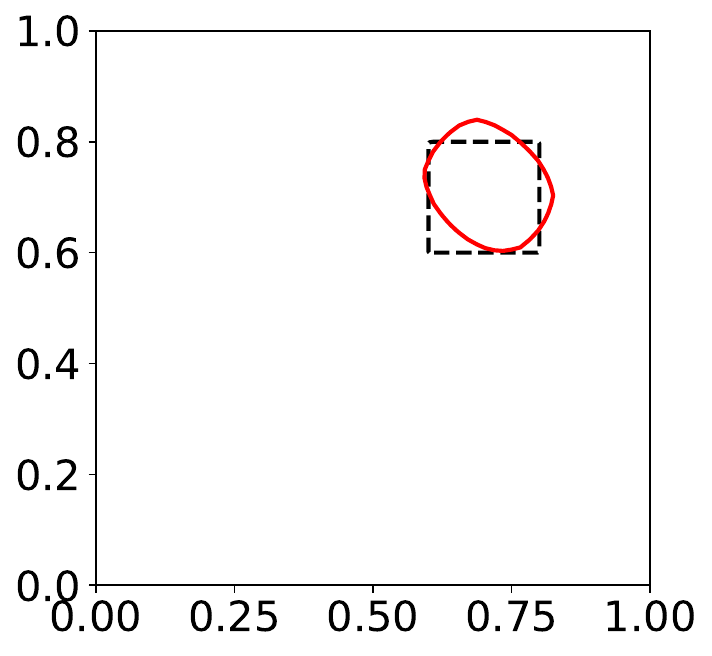}}\hfill
\subfloat{\includegraphics[scale=0.27]{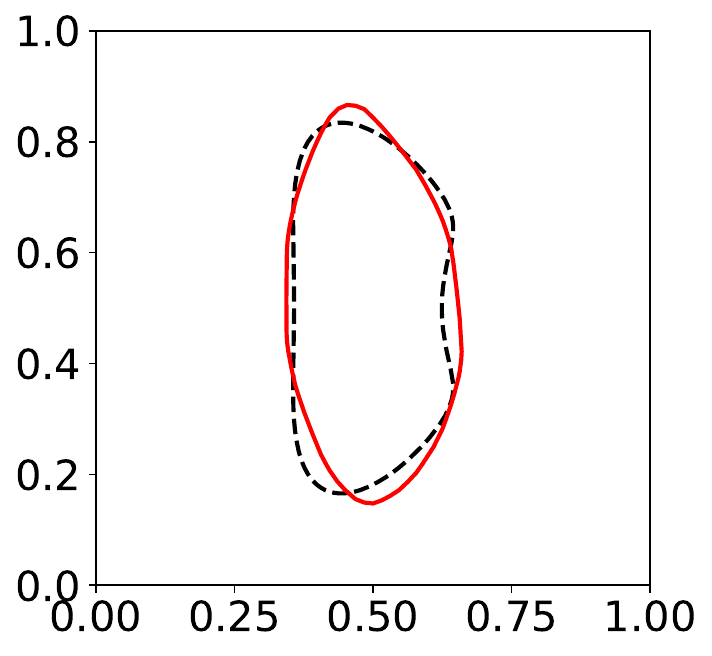}}\hfill
\subfloat{\includegraphics[scale=0.27]{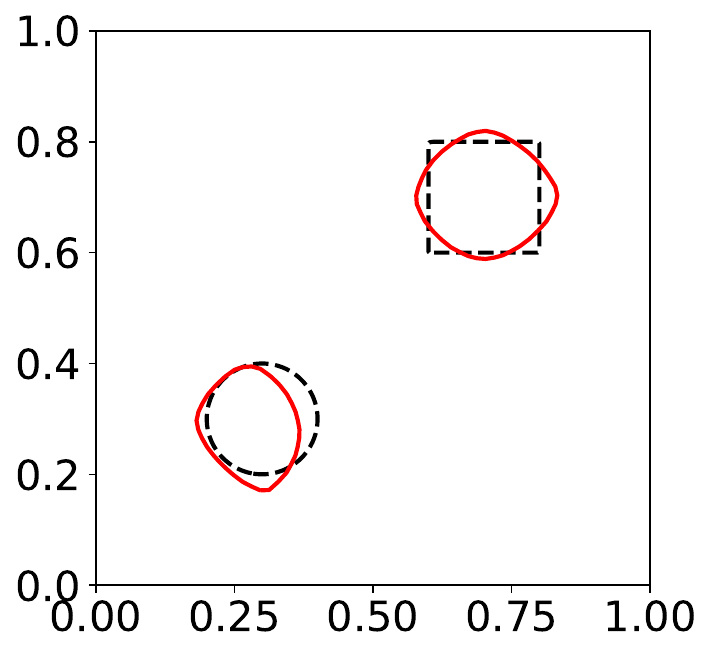}}\\
\subfloat{\includegraphics[scale=0.27]{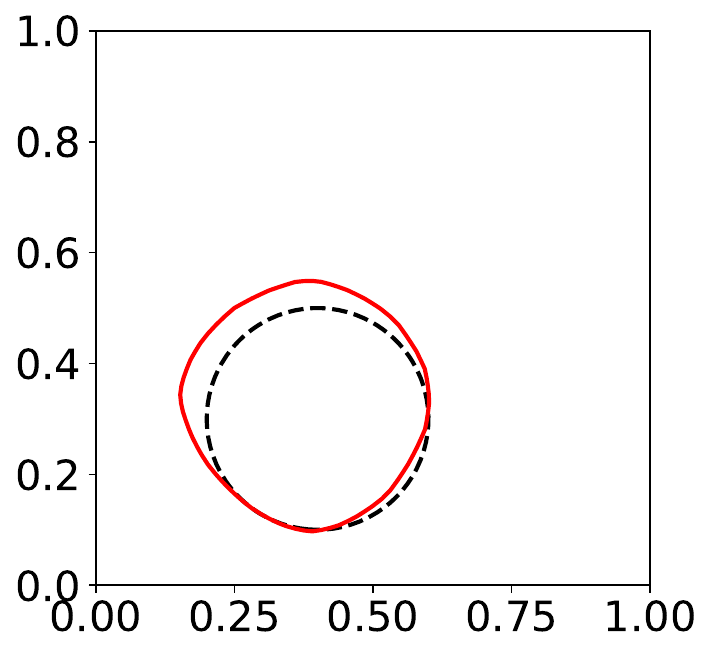}}\hfill
\subfloat{\includegraphics[scale=0.27]{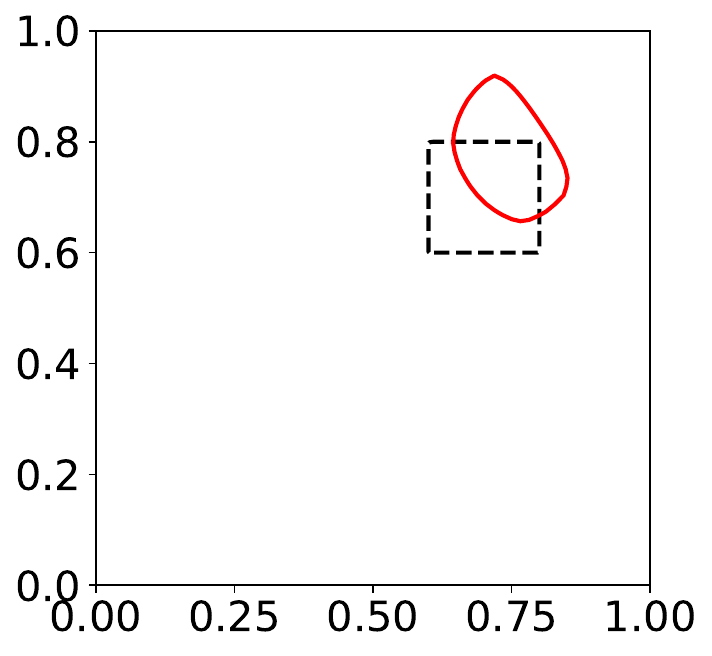}}\hfill
\subfloat{\includegraphics[scale=0.27]{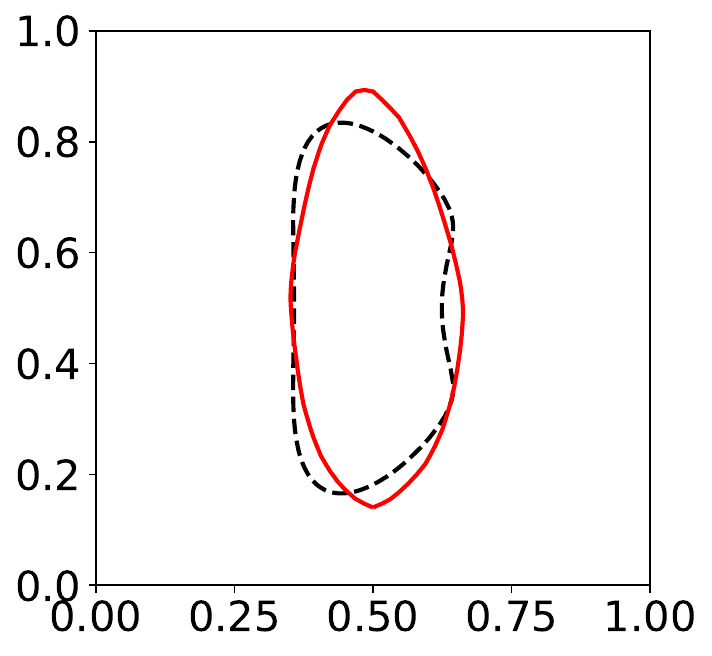}}\hfill
\subfloat{\includegraphics[scale=0.27]{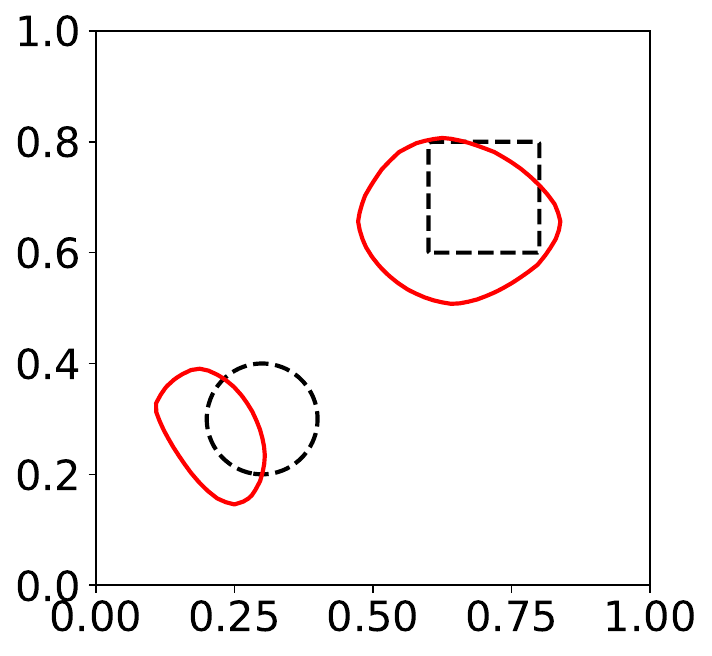}}
\end{center}
\caption{Reconstruction  with the combined monotonicity and levelset  methods. First row: reconstruction with the regularized monotonicity method with
$0\leq a_k\leq \min(\bar c, c_k)$ and $\delta=0.1$ (which is the initialization of the level set method).   The  true inclusion  is the dashed contour.
 Second row:  finale iterate of the level set method in the case of  free noise. Convergence occurs in 40, 60, 48 and 49 iterations (from left to right).
 Third row:  finale iterate of the level set method in the case of  noise level $\eta =0.01$. Convergence occurs in 65, 67, 72 and 62 iterations (from left to right).
 Fourth row:  finale iterate of the level set method in the case  noise level   $\eta =0.03$. Convergence occurs in 70, 75, 89 and 69 iterations (from left to right)}
\label{fig_combine}
\end{figure}

Figure \ref{fig_combine} illustrates the reconstruction achieved through the combined methods. It is evident that this approach yields a superior approximation 
of the solution when compared to using the monotonicity or level set methods individually. Moreover, the reconstruction demonstrates rapid convergence and stability, 
even with varying levels of noise in the data.

By employing monotonicity-based regularization as an initial guess, the convergence rate is enhanced, allowing the combined method to outperform the classical approach in 
terms of speed.

\subsection{Numerical   results for shape and parameter reconstruction}
In this subsection, we present numerical results for the simultaneous reconstruction of the conductivity   $\sigma_1$ and the shape $D$
with five measuremenst $g_k, k=1\ldots 5$ and initial guess  genetated  by the regularized monotonicity method. 

 The algorithm follows a binary search approach: first, the shape $D$ is updated using the level set method, and then, during each iteration,
 the conductivity  $\sigma_1$  is refined using the Newton method by minimizing the convex  functional
\begin{equation}
\varphi(\sigma_1):=\int_\Omega(\sigma_0+(\sigma_1-\sigma_1)\chi_D)\nabla \vert u^g\vert^2\,dx-\int_{\partial\Omega}g f\,ds.
\end{equation}  
Here,  $u^g$  is the solution of the direct problem \eqref{EIT} corresponding to the flux 
$g$, and  $f$represents the boundary trace of the solution 
  to the direct problem \eqref{EIT}, where   $\sigma_1$  and $D$  
are the true parameters.

%--------------------------------------------------------------------------------------------------------%
 \begin{figure}[H]
\begin{center}
\subfloat{\includegraphics[scale=0.35]{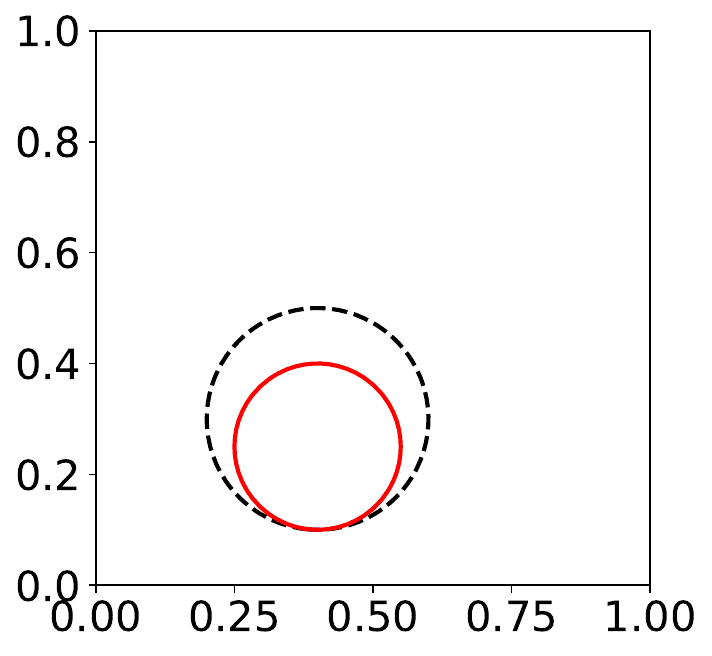}}
\subfloat{\includegraphics[scale=0.35]{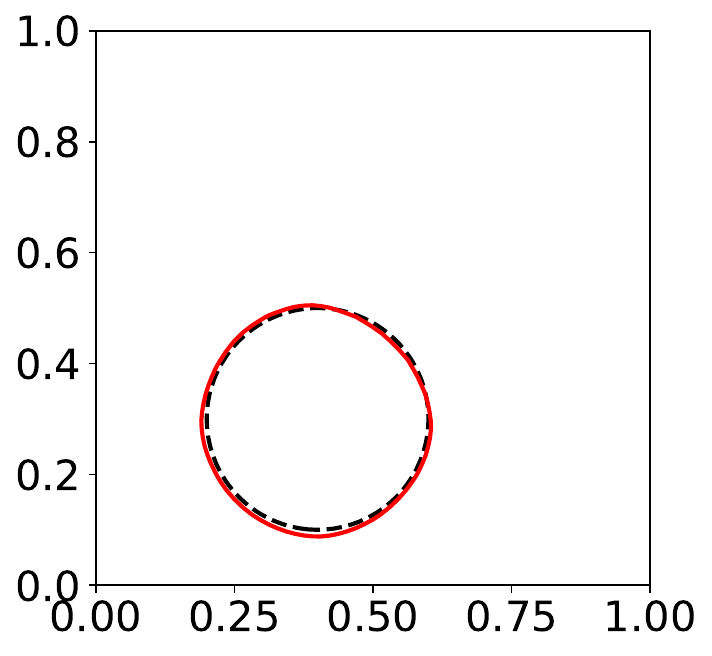}}
\subfloat{\includegraphics[scale=0.35]{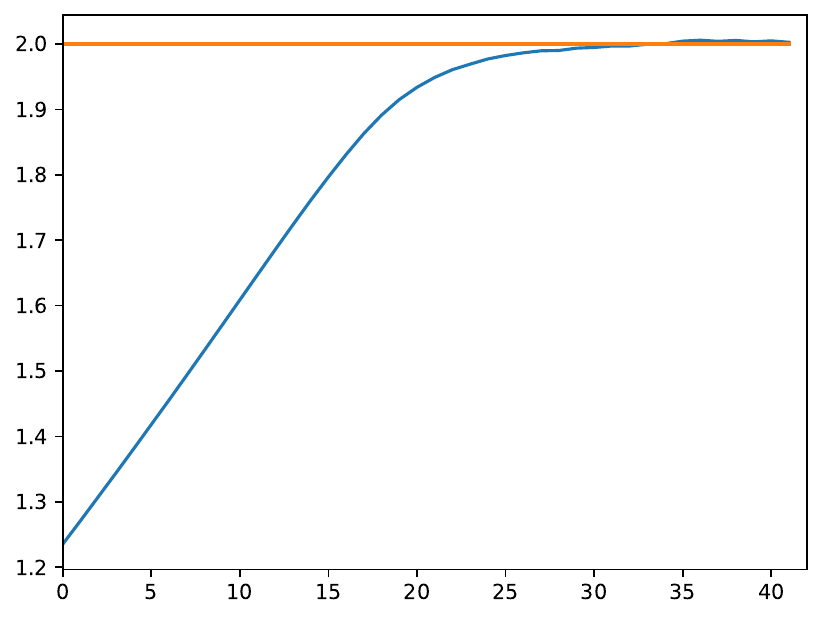}}\\
\subfloat{\includegraphics[scale=0.35]{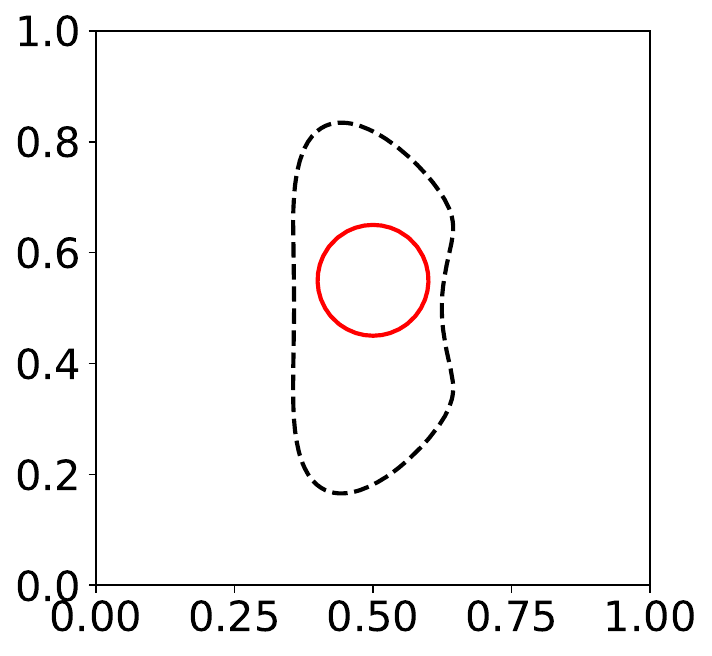}}
\subfloat{\includegraphics[scale=0.35]{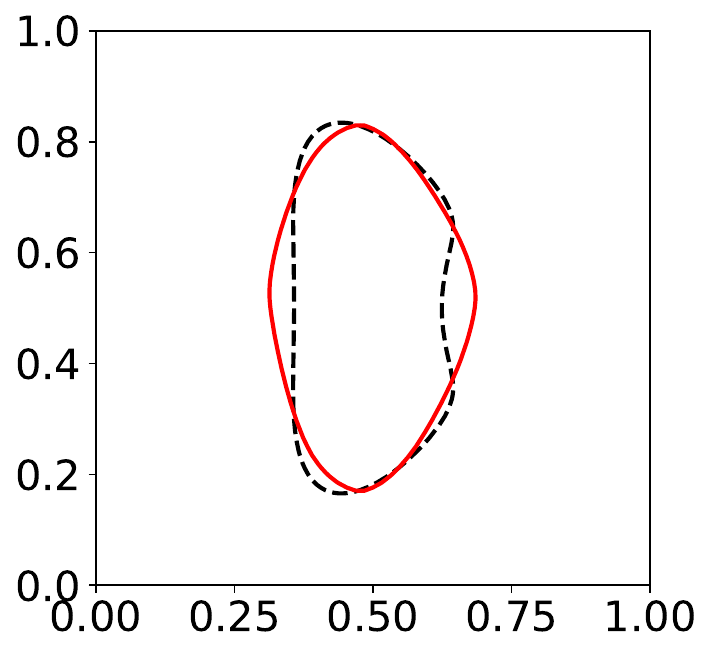}}
\subfloat{\includegraphics[scale=0.35]{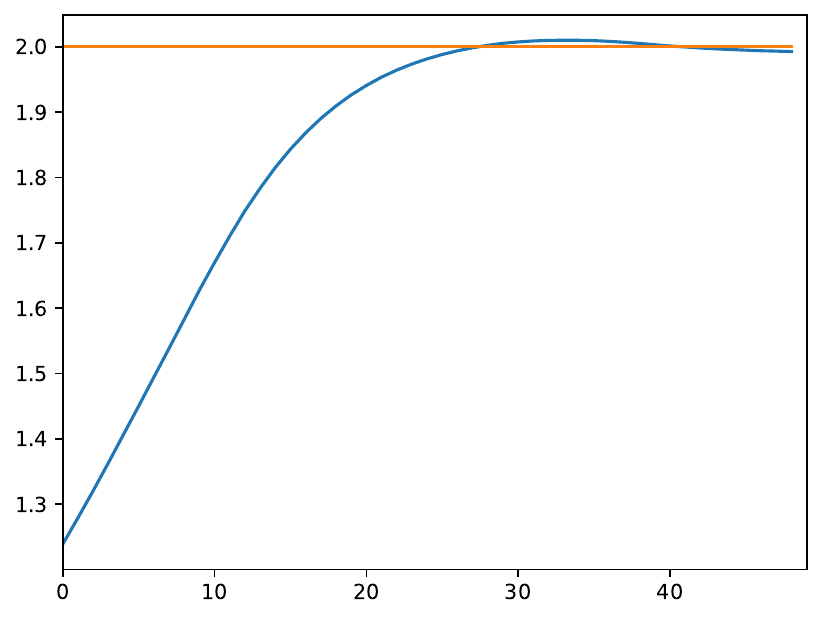}}
\end{center}
\caption{Simultaneous reconstruction of the conductivity $\sigma_1$ and the shape $D$ without noise data $(\eta=0.0)$. First colum: initialization (solid contours) and true inclusion (dashed contour).
 Second  column: reconstruction (solid contours) and true inclusion (dashed contours). Third column:  history of convergence for the  conductivity $\sigma_1$, the orange line 
represents the true value of $\sigma_1$.  Convergence occurs in 43 and 50 iterations(from up to down).}
\label{fig_sigma_shape}
\end{figure}
Figure \ref{fig_sigma_shape} show the simultenous reconstruction of the conductivity $\sigma_1$ and the shape $D$ using $5$ measurements.
In the case of one ball the shape is well  recontruced as well as  the  conductivity  ${\sigma_1}_{\rm approx}=2.002575$.
For the shape described in row 2,  we also get a good reconstruction results.  The approximated conductivity is given by  ${\sigma_1}_{\rm approx}=1.992835$.
%----------------------------------------------------------------------------------------------------------------------------------------%

 \begin{figure}[H]
\begin{center}
\subfloat{\includegraphics[scale=0.35]{Results_sigma_shape/initial_sigma_circle}}
\subfloat{\includegraphics[scale=0.35]{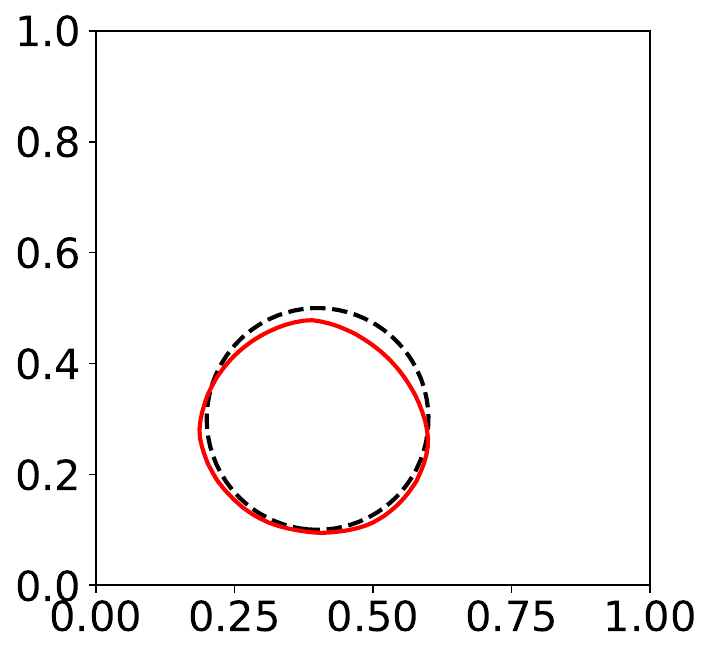}}
\subfloat{\includegraphics[scale=0.35]{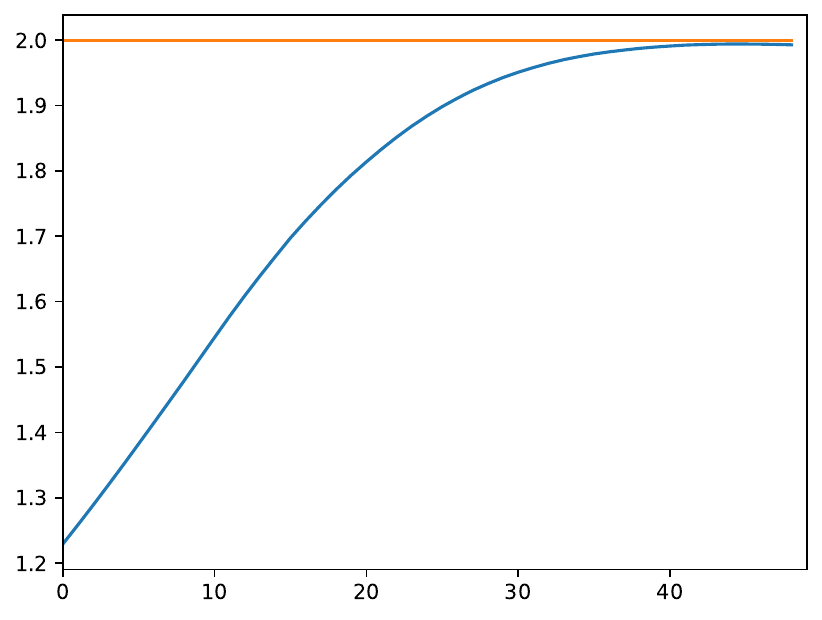}}\\
\subfloat{\includegraphics[scale=0.35]{Results_sigma_shape/initial_sigma_concave}}
\subfloat{\includegraphics[scale=0.35]{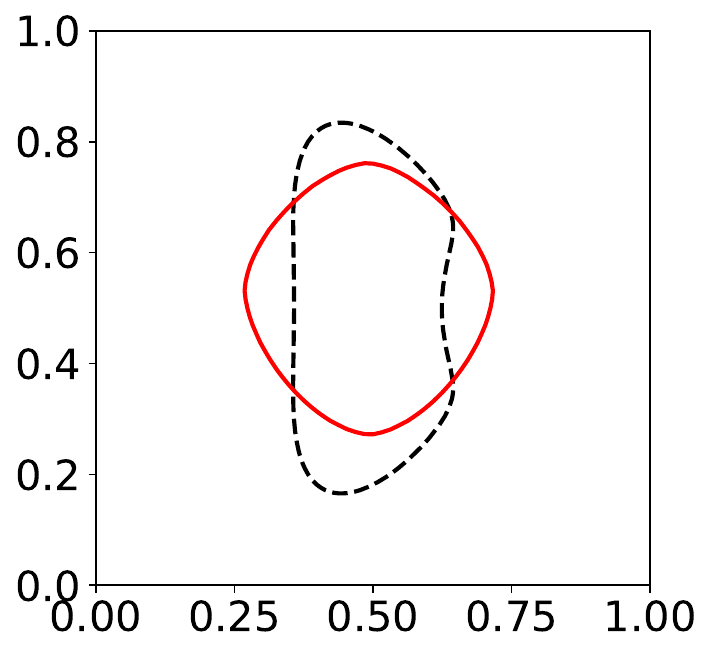}}
\subfloat{\includegraphics[scale=0.35]{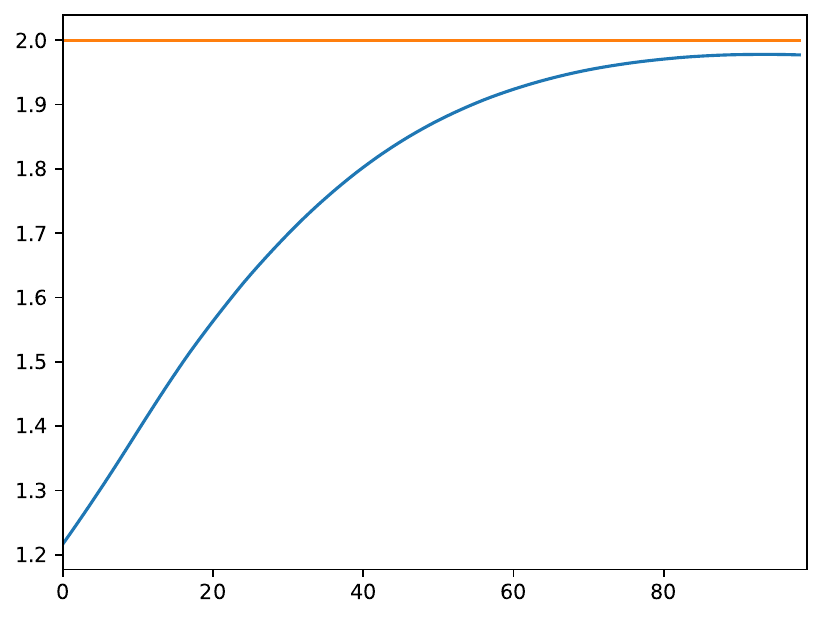}}
\end{center}
\caption{Simultaneous reconstruction of the conductivity $\sigma_1$ and the shape $D$ with noise data $(\eta=0.001)$. First colum: initialization (solid contours) and true inclusion (dashed contour).
 Second  column: reconstruction (solid contours) and true inclusion (dashed contours). Third column:  history of convergence for the  conductivity $\sigma_1$, the orange line 
represents the true value of $\sigma_1$.  Convergence occurs in 50 and 90 iterations(from up to down).}
\label{fig_sigma_shape_noise}
\end{figure}
Figure \ref{fig_sigma_shape_noise} show the simultenous reconstruction of the conductivity $\sigma_1$ and the shape $D$ using $5$ measurements.
In the case of one ball the shape is well  recontruced as well as  the  conductivity  ${\sigma_1}_{\rm approx}= 1.993903$.
For the shape described in  row  2,   the conductivity is well reconstructed compared to the shape.  The approximated conductivity is given by 
 ${\sigma_1}_{\rm approx}=2.0294112$.
\section{Conclusion}

In this paper, we developed a numerical method that combines monotonicity and the level set method to solve a geometric inverse problem.
 We specifically addressed the conductivity problem to apply our numerical scheme.

The effectiveness of the level set method is often influenced by the quality of the initial guess.
 To improve this, we used both the linearized monotonicity method and a regularized monotonicity method to select a suitable initialization. 
Our results show that the regularized monotonicity method provides a more accurate and stable approximation compared to the standard linearized monotonicity method.
 We then used the initial guess from the regularized method to initialize the level set method, and we present the corresponding numerical results.

\bibliographystyle{abbrv}
\bibliography{biblio}

\end{document}